\pdfoutput=1 
\documentclass[english]{article}
\usepackage[T1]{fontenc}
\usepackage[latin9]{inputenc}
\usepackage{animate}
\usepackage[a4paper]{geometry}
\usepackage{color}
\definecolor{note_fontcolor}{rgb}{0.800781, 0.800781, 0.800781}
\usepackage{array}
\usepackage{verbatim}
\usepackage{amsmath}
\usepackage{rotfloat}
\usepackage{enumitem}
\usepackage{url}
\usepackage{geometry}
\usepackage{amstext}
\usepackage{amsthm}
\usepackage{amssymb}   
\usepackage{graphicx}
\usepackage{subcaption}
\usepackage{setspace}
\newtheorem{lemma}{Lemma}
\newtheorem{theorem}{Theorem}
\newtheorem{definition}{Definition}
\usepackage{titlesec}
\titleformat{\paragraph}
{\normalfont\normalsize\bfseries}{\theparagraph}{1em}{}
\titlespacing*{\paragraph}
{0pt}{3.25ex plus 1ex minus .2ex}{1.5ex plus .2ex}
\usepackage{tikz}
\usetikzlibrary{shapes.geometric, arrows}
\definecolor{dodgerblue}{RGB}{30, 144, 255}
\definecolor{dodgerblue_comp}{RGB}{255, 143, 31}
\definecolor{midnightblue}{RGB}{25, 25, 112}
\definecolor{midnightblue_comp}{RGB}{112, 112, 25}
\definecolor{cobaltblue}{RGB}{0, 32, 194}
\definecolor{cobaltblue_comp}{RGB}{194, 162, 0}
\definecolor{hermeshavane}{RGB}{28, 31, 42}
\definecolor{hermesbleunuit}{RGB}{25, 38, 59}
\definecolor{hermesbleunuitmed}{RGB}{25, 38, 59}
\definecolor{hermesbleunuitlight}{RGB}{25, 38, 59}
\definecolor{hermeshavanebrown}{RGB}{91, 58, 41}

\tikzstyle{startstop} = [rectangle, rounded corners, minimum width=3cm, minimum height=1cm,text centered, draw=black, fill=red!39]
\tikzstyle{io} = [trapezium,rectangle, rounded corners, minimum width=3cm, minimum height=1cm,text centered, text width =3cm, draw=white, fill=hermesbleunuit, text = white]
\tikzstyle{arrow} = [thick,->,>=stealth]
\tikzstyle{process} = [rectangle, rounded corners, minimum width=1cm, minimum height=1.5cm, text centered, text width =5.2cm, draw=white, fill=hermesbleunuit, text = white, fill opacity = 0.92]
\tikzstyle{process2} = [rectangle, rounded corners, minimum width=3cm, minimum height=1.5cm, text centered, text width =5.2cm, draw=white, fill=hermesbleunuit, text = white, fill opacity = 0.85]

\tikzstyle{io2} = [trapezium,rectangle, rounded corners, minimum width=3cm, minimum height=1cm,text centered, text width =3cm, draw=white, fill=hermeshavanebrown, text = white]
\tikzstyle{processalt} = [rectangle, rounded corners, minimum width=1cm, minimum height=1.5cm, text centered, text width =5.2cm, draw=white, fill=hermeshavanebrown, text = white, fill opacity = 0.92]
\tikzstyle{processalt2} = [rectangle, rounded corners, minimum width=6cm, minimum height=1.5cm, text centered, text width =5.2cm, draw=white, fill=hermeshavanebrown, text = white, fill opacity = 0.85]
\usepackage{listings}
\usepackage{float}
\usepackage[toc,page]{appendix}

\makeatletter

\usepackage[numbers,square]{natbib}

\makeatother
\title{Pseudo-Poisson Distributions with Nonlinear Conditional Rates%
\thanks{For full functionality of the animated figures presented in this document, please view this PDF in \textbf{Adobe Acrobat Reader}. Other PDF viewers (including browser-based viewers, 
Preview) may render the figures statically.}}

\author{
  Jared N. Lakhani\\
  \textit{Department of Statistical Sciences, University of Cape Town}\\
  \texttt{lkhjar001@myuct.ac.za}
}
\date{} 

\usepackage[hidelinks]{hyperref}

\begin{document}
\maketitle

\begin{abstract}
\citet{arnold2021statistical} claim that the bivariate pseudo-Poisson distribution is well suited to bivariate count data with one equidispersed and one overdispersed marginal, owing to its parsimonious structure and straightforward parameter estimation. In the formulation of \citet{leiter1973some}, the conditional mean of \(X_2\) was specified as a function of \(X_1\); \citet{arnold2021statistical} subsequently augmented this specification by adding an intercept, yielding a linear conditional rate. A direct implication of this construction is that the bivariate pseudo-Poisson distribution can represent only positive correlation between the two variables. This study generalizes the conditional rate to accommodate negatively correlated datasets by introducing curvature. This augmentation provides the additional benefit of allowing the model to behave approximately linear when appropriate, while adequately handling the boundary case \((x_1,x_2)=(0,0)\). According to the Akaike Information Criterion (AIC), the models proposed in this study outperform \citet{arnold2021statistical}'s linear models.
\end{abstract}

\section{Introduction}
We review the concept of the bivariate pseudo-Poisson distribution as discussed in \citet{arnold2021statistical}. 
\begin{definition}
    A pair of variables ($X_1, X_2$) is said to have a bivariate pseudo-Poisson distribution if there exists a positive constant $\lambda_1$ such that:
    \begin{align}
        X_1 \sim \text{Poisson}(\lambda_1),\nonumber
    \end{align}
    and a function $\lambda_2:\{0,1, 2\ldots \} \rightarrow(0, \infty)$ such that, for every non-negative integer $x_1$:
    \begin{align}
        X_2\mid X_1 \sim \text{Poisson}(\lambda_2(x_1)).\nonumber
    \end{align}
\end{definition}
Motivated by Example~1 in \citet{arnold2023pseudo}, we specify the conditional rate as $\lambda_2(x_1)=\delta+\beta\,F(x_1;\boldsymbol{\theta})$, where \(\delta\ge0\) is an intercept, $\beta\in \mathbb{R}$ is an amplitude parameter, and \(F(x_1;\boldsymbol{\theta})\) is a parametric distribution function on \([0,\infty)\) (with \(\lambda_2(x_1)\ge0\) for all \(x_1\)). This parameterization permits negative correlation between \(X_1\) and \(X_2\) when \(\beta<0\) and \(\delta\ge \max\{-\beta, 0\}\) to ensure \(\lambda_2(x_1)\ge0\). By Theorem~\ref{thm:rho_postive}, the sign of \(\operatorname{Corr}(X_1,X_2)\) is determined by the monotonicity of \(\lambda_2(x_1)\). Since \(F(x_1;\boldsymbol{\theta})\) is always non-decreasing, the monotonicity of \(\lambda_2(x_1)\) is dependent on the sign of $\beta$. Hence, if \(\lambda_2\) is strictly increasing in \(x_1\) ($\beta> 0$), the correlation is positive; if \(\lambda_2\) is strictly decreasing ($\beta< 0$), the correlation is negative.

\begin{theorem} \label{thm:rho_postive}
Given a pseudo-Poisson form $X_1 \sim \text{Poisson}(\lambda_1)$, $X_2 \mid X_1 \sim \text{Poisson}\left(\lambda_2(x1) \right)$ with $\lambda_1, \lambda_2(x_1) > 0$ for $x_1, x_2 \in \{0, 1, 2\ldots \}$: if  $\lambda_2(x_1)$ is strictly increasing with respect to $x_1$, then $\rho(X_1, X_2) > 0$. 

\begin{proof}
    We note $E(X_1) = \alpha$, $E(X_2) = E_{X_1}\{E_{X_2\mid X_1}(X_2 \mid X_1 = x_1) \} = E_{X_1}\{\lambda_2(X_1) \}$. Furthermore, $E(X_1X_2) = E_{X_1}\{ X_1 E_{X_2\mid X_1}\{X_2 \mid X_1 = x_1 \} \} = E_{X_1}\{X_1\lambda_2(X_1)  \}$. Now we have:
    \begin{align}
        Cov(X_1, X_2) &= E(X_1X_2) - E(X_1)E(X_2) \nonumber \\
        & = E_{X_1}\{X_1\lambda_2(X_1)  \} - \lambda_1 E_{X_1}\{\lambda_2(X_1) \} \nonumber\\
        & = \sum_{i = 0}^\infty i \lambda_2(i)\frac{e^{-\lambda_1}\lambda_1^i}{i!} - \sum_{i = 0}^\infty \lambda_1 \lambda_2(i)\frac{e^{-\lambda_1}\lambda_1^i}{i!} \nonumber\\
        & = \sum_{i = 1}^\infty i \lambda_2(i)\frac{e^{-\lambda_1}\lambda_1^i}{i!} - \sum_{i = 0}^\infty \lambda_1 \lambda_2(i)\frac{e^{-\lambda_1}\lambda_1^i}{i!} \nonumber \\
         & = \sum_{i = 0}^\infty (i+1) \lambda_2(i+1)\frac{e^{-\lambda_1}\lambda_1^{i+1}}{(i+1)!} - \sum_{i = 0}^\infty \lambda_1 \lambda_2(i)\frac{e^{-\lambda_1}\lambda_1^i}{i!} \nonumber \\
       & = \sum_{i = 0}^\infty \lambda_1 \lambda_2(i+1)\frac{e^{-\lambda_1}\lambda_1^{i}}{i!} - \sum_{i = 0}^\infty \lambda_1 \lambda_2(i)\frac{e^{-\lambda_1}\lambda_1^i}{i!} \nonumber \\
       & = \sum_{i = 0}^\infty (i+1) \frac{e^{-\lambda_1}\lambda_1^{i}}{(i+1)!} \left(\lambda_2(i+1) - \lambda_2(i) \right) \nonumber 
    \end{align}
Since $\lambda_2(i+1) > \lambda_2(i)$, it follows that ${Cov}(X_1, X_2) > 0$, and consequently $\rho(X_1, X_2) > 0$.
Conversely, if $\lambda_2(x_1)$ is strictly decreasing such that $\lambda_2(i+1) < \lambda_2(i)$, then ${Cov}(X_1, X_2) < 0$, implying $\rho(X_1, X_2) < 0$.
When $\lambda_2(x_1)$ is constant, $\rho(X_1, X_2) = 0$.
\end{proof}
\end{theorem}
Additionally, we offer Theorem \ref{thm:poisson_var}, merely to show an alternative way of proving Theorem $6.3$ in \citet{arnold2021statistical}, claiming the bivariate pseudo-Poisson distribution is always over-dispersed.

\begin{theorem} \label{thm:poisson_var}  Given a pseudo-Poisson form $X_1 \sim \text{Poisson}(\lambda_1)$, $X_2 \mid X_1 \sim \text{Poisson}\left(\lambda_2(x_1) \right)$, $Var(X_1) = E(X_1)$ and $Var(X_2) \geq E(X_2)$ given $\lambda_1, \lambda_2(x_1) > 0$ for $x_1, x_2 \in \{0, 1, 2\ldots \}$. 
\end{theorem}

\begin{proof}
The first identity follows directly from the properties of the Poisson distribution. Since $X_1 \sim \text{Poisson}(\lambda_1)$, we have $Var(X_1) = E(X_1) = \lambda_1$. Using $Var_{X_2|X_1}(X_2|X_1) = E_{X_2|X_1}(X_2|X_1)$ from the properties of a Poisson distribution since $X_2\mid X_1$ is Poisson distributed, we have:
\begin{align}
    Var(X_2) &=E_{X_1} \{Var_{X_2|X_1}(X_2|X_1) \} + Var_{X_1}\{E_{X_2|X_1}(X_2|X_1)    \}
    \nonumber\\
    & = E_{X_1} \{E_{X_2|X_1}(X_2|X_1) \} + Var_{X_1}\{E_{X_2|X_1}(X_2|X_1)    \} \nonumber \\
     & = E(X_2) + Var_{X_1}\{\lambda_2(X_1) \} \nonumber \\
     & \geq E(X_2) \nonumber
\end{align}
where the inequality follows since variance is always non-negative.
\end{proof}
In this study, the distribution function \(F(x_1;\boldsymbol{\theta})\) is specified to belong to either the exponential or Lomax family. Importantly, these two were chosen because closed-form expressions for their moments can be obtained, with infinite sums simplified either via the exponential Maclaurin series or by using hypergeometric functions. The study further segregates these full models - models with full parameterization - into sub-models in which certain parameters are nullified (by fixing them at $0$ or $1$), inspired by \citet{arnold2021statistical}. The study specifies method-of-moments estimators (m.m.e.'s) either through implicit relations or explicitly stated equations, examines the correlations each model can accommodate, and includes likelihood-ratio tests to compare the sub-models with their respective full models. Additionally, simulated data are provided for enrichment, after which two datasets (also used in \citet{arnold2021statistical}) are analysed as applications of the models. Furthermore, in terms of model fit according to the Akaike Information Criterion (AIC), this study's models outperform \citet{arnold2021statistical}'s linear models. We attribute this to the curvature of the conditional rate \(\lambda_2(x_1)\), where this extends to their added ability to model $(x_1, x_2) = (0, 0)$ observations adequately. This study's models also allow for negatively correlated bivariate datasets - although no such data is analysed here.

\section{Exponential}
We restrict $F(x_1; \boldsymbol{\theta})$, with $\boldsymbol{\theta} = \gamma$, to be the exponential distribution function. Thus, we have:
\begin{eqnarray}
	X_{1} &\sim & \text{Poisson}(\alpha)  \label{eq: exp Function 1} \\ 
	X_{2} \mid X_{1} &\sim & \text{Poisson}\left(\delta + \beta \left(1 - e^{- \gamma x_1}\right)\right),  \label{eq: exp Function 2}
\end{eqnarray}
where the rate parameter $\gamma > 0$ governs the rate of change (gradient) of the regression function. Later, we re-parameterize via  $e^{-\gamma} = \nu$. The parameters satisfy $\alpha, \gamma > 0$. To ensure that $\lambda_2(x_1) > 0$ for all $x_1 > 0$, it is necessary that $\beta \in \mathbb{R} \setminus \{0\}$ with $\delta \geq \max\{-\beta, 0\}$ - which corresponds to $\delta \geq -\beta$ in the case of negative correlation $(\beta < 0)$ and $\delta \geq 0$ when the correlation is positive $(\beta > 0)$. We ensure $\beta \neq 0$ here so to ensure an independence model is not created - that is, where $X_2$ is no longer conditioned on $X_1$.\\

We also examine various sub-models that arise from the exponential model by constraining specific parameters. These include removing the amplitude control - thereby fixing positive and negative correlation by setting $\beta = 1$ and $\beta = -1$ respectively, fixing the rate parameter to $\gamma = 1$, or jointly constraining both. In addition, we consider a no-intercept model by setting $\delta = 0$, which by construction precludes the possibility of negative correlation. We further note, that for the no-intercept model, Case III: $\delta = 0$, when $x_{1}=0$, the conditional distribution of $x_{2}$ degenerates at zero, implying that $x_{2}=0$ with probability one. Consequently, any empirical observation with $x_{1}=0$ and $x_{2}>0$ is assigned zero probability under the model, which yields a likelihood of zero and indicates model misspecification for such observations. Figure \ref{fig:exp flow} provides a schematic illustration of the resulting sub-models.

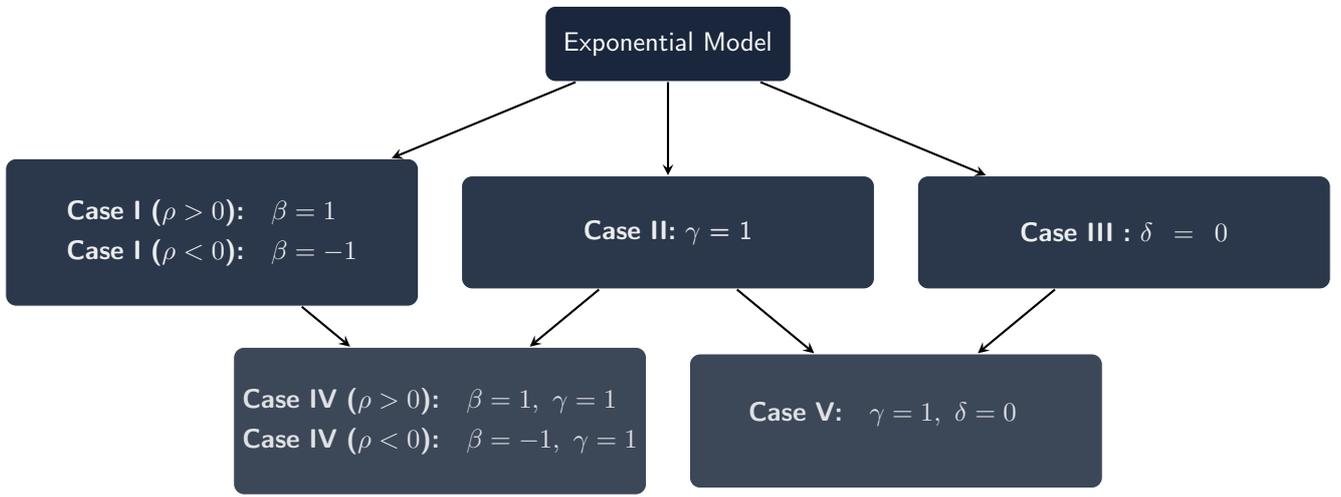
\begin{figure}[H] \centering
	\begin{tikzpicture}[node distance=2cm]
		\node (exp) [io] {Exponential Model};
		\node (expcase1) [process, below of = exp,  xshift = -6cm, yshift = -0.5cm]{\[
        \begin{aligned}
		&\textbf{Case I ($\rho > 0$):} \quad \beta = 1 \\ 
        &\textbf{Case I ($\rho < 0$):} \quad \beta = -1 
        \end{aligned}
        \]}
        ;
		\node (expcase2) [process, below of = exp, yshift = -0.5cm]{\textbf{Case II:} $\gamma$ = 1\\ };
	\node (expcase3) [process2, below of = expcase1, xshift = 3cm, yshift = -0.5cm] {\[
  \begin{aligned}
    &\textbf{Case IV ($\rho > 0$):} \quad \beta = 1,\ \gamma = 1 \\
    &\textbf{Case IV ($\rho < 0$):} \quad \beta = -1,\ \gamma = 1
  \end{aligned}\]}
     ;
\node (expcase4) [process, below of = exp, xshift = 6cm, yshift = -0.5cm]{\textbf{Case III :} $\delta = 0$};
\node (expcase5) [process2, below of = expcase4, xshift = -3cm, yshift = -0.5cm] {
\[
  \begin{aligned}
    &\textbf{Case V:} \quad \gamma = 1,\ \delta = 0
  \end{aligned} \] }
     ;
		\draw [arrow] (exp) -- (expcase1);
		\draw [arrow] (exp) -- (expcase2);
		\draw [arrow] (expcase1) -- (expcase3);
		\draw [arrow] (expcase2) -- (expcase3);
		\draw [arrow] (exp) -- (expcase4);
        \draw [arrow] (expcase4) -- (expcase5);
		\draw [arrow] (expcase2) -- (expcase5);
	\end{tikzpicture}
	\caption{Flow diagram of exponential sub-models.}
    \label{fig:exp flow}
\end{figure}

\subsection{Probability mass functions}
Following from the definition of a Poisson distribution:
\begin{eqnarray}
	P(X_{1}) &=& \frac{e^{-\alpha}\alpha^{x_{1}}}{x_{1}!}, \nonumber\\
	P(X_{2}|X_{1}) &=& \frac{e^{-\delta - \beta (1 -e^{- \gamma x_1})} \{\delta + \beta (1 -e^{- \gamma x_1})\}^{x_2}       } {x_{2}!} , \nonumber \\    
	P(X_{1}, X_{2}) &=& P(X_{1})P(X_{2}|X_{1}) \nonumber\\  
	&=& \frac{e^{-\alpha}\alpha^{x_{1}}}{x_{1}!} \frac{e^{-\delta - \beta (1 -e^{- \gamma x_1})} \{\delta + \beta (1 -e^{- \gamma x_1})\}^{x_2}       } {x_{2}!},\label{eq:P(X1, X2) exp}
\end{eqnarray}
where $x_{1}$, $x_{2}$ $\in$ $\{0, 1, 2...\}$, $\alpha, \gamma > 0$ and $\beta \in \mathbb{R}\setminus \{0\}$ with $\delta \geq \max(-\beta, 0)$.
\subsection{Moments} \label{sec: moments exp}
After substituting $e^{\alpha} = \mu$ and $e^{-\gamma} = \nu$, we note:
\begin{eqnarray}
	E(X_1) &=& \alpha \label{eq: E(X1) exp}, \\
	E(X_2)= E_{X_{1}}\{E(X_2|X_1 = x_1)\}  &=& \delta +  \beta - \beta E_{X_1}(\nu^{X_1}) \nonumber \\
	 &= & \delta + \beta ( 1 - \mu^{\nu - 1} ) ,\label{eq: E(X2) exp} 
\end{eqnarray}
Since
\begin{eqnarray}
	E_{X_1}(\nu^{X_1}) &=& \sum_{i = 0}^{\infty}	\nu^i \ \frac{\mu^{-1} {\alpha} ^i} {i!} \nonumber \\
	&=& \mu^{-1} \sum_{i = 0}^{\infty} \frac{ (\alpha\nu)^i } {i!}    \nonumber \\
	&=&	\mu^{-1} e^{\alpha\nu}     \nonumber \\
	&=& \mu ^ {\nu - 1}   \label{eq: E(e) exp}
\end{eqnarray}
We note, since $0<\nu<1$ $\implies$ $\frac{1}{\mu} < \mu^{\nu-1} < 1$ confirming $E(X_2) > 0$. Now
\begin{eqnarray}
	Var(X_1)&=& \alpha, \label{eq: Var(X1) exp} \\
	Var_{X_2|X_1}(X_2|X_1) &=& \delta+ \beta(1 - \nu^{X_1}) \nonumber \\
	&=& E_{X_2|X_1}(X_2|X_1) ,\label{eq: Var(X2|X_1) exp} 
\end{eqnarray}
Hence, from Equations \ref{eq: E(X2) exp}, \ref{eq: E(e) exp}, \ref{eq: Var(X2|X_1) exp}:
\begin{eqnarray}
	Var(X_2) &=& E_{X_1} \{Var_{X_2|X_1}(X_2|X_1) \} + Var_{X_1}\{E_{X_2|X_1}(X_2|X_1)    \} \nonumber \\
	&=& E(X_2) + {\beta}^{2} Var_{X_1}(\nu^{X_1}) \nonumber \\
	&=& E(X_2) + {\beta}^{2} \{ E_{X_{1}}(\nu^{2X_1}) - \{E_{X_1}(\nu^{X_1}) \}^2    \}  \nonumber \\
	&=&  \beta(1 -\mu^{\nu - 1}) + \delta +  \beta^2 (\mu^{\nu^2 - 1} -\mu^{2\nu - 2}).\label{eq: Var(X2)}
\end{eqnarray}
Note $0<\nu<1 \implies \mu^{\nu^2 - 1} > \mu^{2\nu-2} \implies Var(X_2) > E(X_2)$, which is a stricter version of Theorem \ref{thm:poisson_var}. Furthermore from Equation \ref{eq: E(X1) exp}:
\begin{eqnarray}
	E(X_1 X_2) = E_{X_1} \{ E(X_1 X_2 | X_1 )  \}  &=& E_{X_1} \{ X_1 E( X_2 | X_1 )  \} \nonumber \\
	&=& E_{X_1}\{ X_1\left(\delta +  \beta (1 - e^{-\gamma X_1})  \right)\} \nonumber \\
	&=& \alpha (\delta+ \beta ) - \beta E_{X_1} ( X_1 \nu^{X_1})  \nonumber \\
	&=& \alpha\left( \delta+\beta(1 - \nu\mu^{\nu - 1}) \right)\label{eq: E(X1X2) exp} , 
\end{eqnarray}
since
\begin{eqnarray}
	E_{X_1} ( X_1 \nu^{X_1} ) &=&  \sum_{i = 0}^{\infty} i \nu^i \  \frac{ \mu^{-1} {\alpha}^i} {i!} \nonumber \\
	&=& \mu^{-1} \sum_{i = 1}^{\infty} \alpha\nu \frac{ (\alpha\nu)^{(i-1)}}  {(i-1)!} \nonumber \\
	&=& \mu^{-1} \alpha \nu \sum_{n = 0}^{\infty}
	\frac{ \left(\alpha \nu\right)^{n}}  {n!} \nonumber \\
	&=& \mu^{-1} \alpha \nu e^{\alpha \nu } \nonumber \\
	&=& \alpha\nu\mu^{\nu  - 1}. \nonumber
\end{eqnarray}
Now from Equations \ref{eq: E(X1) exp}, \ref{eq: E(X2) exp}, \ref{eq: E(X1X2) exp}, the covariance between $X_1$ and $X_2$ is:
\begin{eqnarray}
	Cov(X_1, X_2) &=& E(X_1 X_2) - E(X_1)E(X_2) \nonumber \\
	&=& \alpha\beta(1 -\nu)\mu^{\nu -1 } \label{eq: Cov(X1X2) exp}.
\end{eqnarray}
We note that since $0<\nu<1 \implies \alpha (1-\nu)\mu^{\nu - 1} > 0$. Consequently, the sign of the covariance - and thus the direction of the correlation between $X_1$ and $X_2$ - is determined solely by the sign of $\beta$, as previously noted. Now the associated correlation using Equations \ref{eq: Var(X1) exp}, \ref{eq: Var(X2)} and \ref{eq: Cov(X1X2) exp}:
\begin{eqnarray}
	\rho(\alpha, \beta, \gamma, \delta) &=& \frac {Cov(X_1 X_2)}{\sqrt{Var(X_1) Var(X_2)}} \nonumber \\
	&=& \frac{\sqrt{\alpha}\beta(1 -\nu)\mu^{\nu - 1} }{\sqrt{\beta^2(\mu^{\nu^2 - 1} - \mu^{2\nu -2}) + \beta(1 - \mu^{\nu - 1}) + \delta }         }. \label{eq: rho exp} 
\end{eqnarray}
\subsubsection{Method of moments estimation}
Suppose data in the form of $\textbf{X}^{(1)}, \textbf{X}^{(2)}, ..., \textbf{X}^{(n)}$ (where for each $i = 1, \ldots, n$, $\textbf{X}^{(i)} = (X_{1i}, X_{2i})^{T} )$ are obtained, which are independent and identically distributed according to Equations \ref{eq: exp Function 1} - \ref{eq: exp Function 2}. If $M_{1}, M_2 > 0$, consistent asymptotically normal method of moment estimates (m.m.e's) are acquired.\\
Define:
\begin{eqnarray}
	M_{1} &=& \frac{1}{n}\sum_{i = 1}^{n} x_{1i}, \nonumber\\
	M_{2} &=& \frac{1}{n}\sum_{i = 1}^{n} x_{2i},\nonumber\\
	S_{12}&=& \frac{1}{n}\sum_{i = 1}^{n}(x_{1i} - M_{1} ) ( x_{2i} - M_{2} ), \nonumber \\
	S_{22} &=&  \frac{1}{n}\sum_{i = 1}^{n}(x_{2i} - M_{2} )^2. \nonumber
\end{eqnarray}	
Now equating Equation \ref{eq: E(X1) exp} to its sample mean:
\begin{eqnarray}
	\tilde{\alpha} &=& M_{1},\label{eq: mme alpha exp} \\
	e^{\tilde{\alpha}} &=& e^{M_{1}}  = \tilde{\mu}, \nonumber
\end{eqnarray}	
where Equation \ref{eq: mme alpha exp} is the applicable m.m.e for all models. Consider the remaining moments, $\tilde{\beta}$, $\tilde{\gamma}$ and $\tilde{\delta}$ for each model separately: \\
\paragraph{Full exponential model:}
Taking the quotient of Equation \ref{eq: Cov(X1X2) exp} squared and the difference between Equation \ref{eq: Var(X2)} and Equation \ref{eq: E(X2) exp}, and equating the result to the solution obtained from identical operations on the associated sample moments, an implicit equation w.r.t ${\tilde{\nu}}$ (hence ${\tilde{\gamma}}$) is obtained, to be solved numerically:
\begin{eqnarray}
	\frac{\tilde{\alpha}^2(\tilde{\nu} -1)^2}{\tilde{\mu}^{(\tilde{\nu} - 1)^2} -1 } = \frac{S_{12}^2}{S_2- M_2}, \label{eq: nu exp}
\end{eqnarray}
and to ensure the expression remains real-valued we obtain the condition:
\begin{eqnarray}
    S_2 &\neq & M_2,\nonumber
\end{eqnarray}
although we already elucidated that $Var(X_2) > E(X_2)$ for $\beta \neq 0$. Since $\forall x_{1}, x_{2} \in \{0, 1, 2...\}$,  $\tilde{\gamma} > 0 \implies  0<\tilde{\nu}<1$. Hence, since Equation \ref{eq: nu exp} is strictly increasing for $0<\tilde{\nu}<1$ from Theorem \ref{thm: f(v) strict inc}  (see Appendix \ref{app: thms} for proof), we may use $\lim_{\tilde{\nu} \to 0^+}$ and $\lim_{\tilde{\nu} \to 1^-}$ to obtain the lower and upper bounds, respectfully, of the left side of Equation \ref{eq: nu exp}, hence $\frac{S_{12}^2}{S_2-M_2}$. Therefore:
\begin{eqnarray}
	\frac{M_1^2}{e^{M_1} - 1} &<& \frac{{S_{12}}^2}{{S_2 - M_2}}   <  M_1  \label{eq: full bounds 2}.
\end{eqnarray}
Figure \ref{fig: nu mme full exp} illustrates $\tilde{\nu}$ as a function of $\tilde{\alpha} =M_1$ and $\frac{S_{12}^2}{S_2 - M_2}$. Importantly, due to the bounds imposed by $M_1$ on $\frac{S_{12}^2}{S_2 - M_2}$ in Equation \ref{eq: full bounds 2}, certain combinations of $M_1$, $S_{2}$, $M_{2}$ and $S_{12}$ values may lead to the non-existence of $\tilde{\gamma}$ (hence $\tilde{\nu}$). This behavior is also evident from the patterns observed in Figure \ref{fig: nu mme full exp}. 
\begin{figure}[H]
    \centering
    \includegraphics[width=0.5\linewidth]{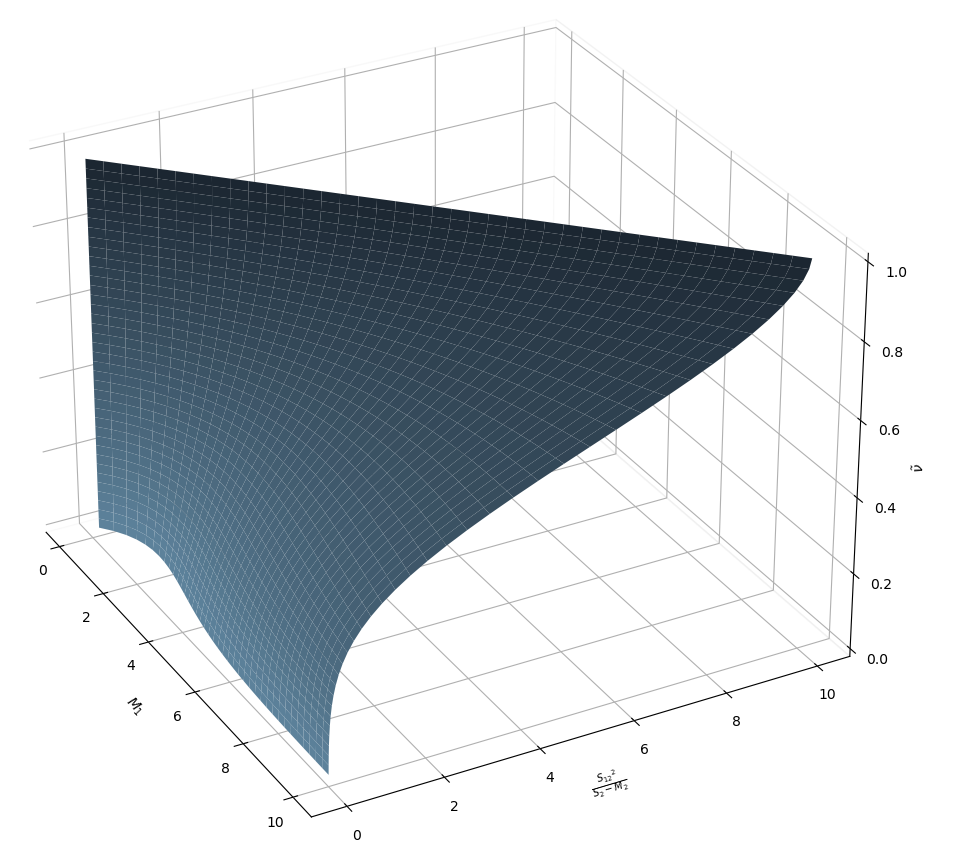}
    \caption{$\tilde{\nu}$ vs $M_1$ and $\frac{S_{12}^2}{S_2 - M_2}$ as per Equation \ref{eq: nu exp}.}
    \label{fig: nu mme full exp}
\end{figure}

We then solve for $\tilde{\beta}$ by equating Equation \ref{eq: Cov(X1X2) exp} to the sample covariance to obtain:
\begin{equation}
     \tilde{\beta} = \frac{S_{12}}{\tilde{\alpha}(1 - \tilde{\nu})\tilde{\mu}^{\tilde{\nu}-1}} .\label{eq: mme beta exp}
\end{equation}
Finally, solve for $\tilde{\delta}$ by merely equating Equation \ref{eq: E(X2) exp} to the sample mean:
\begin{eqnarray}
	\tilde{\delta} &=& M_2 - \tilde{\beta}(1 - \tilde{\mu}^{\tilde{\nu} - 1}).
	\label{eq: mme delta exp}
\end{eqnarray}
We note, that the m.m.e's of Case 3: $\delta = 0$  may be derived analogously to the approach previously outlined, with the exception of excluding the estimator for $\delta$. Additionally, we observe that Equation~\ref{eq: nu exp} may be rewritten as $\frac{\zeta}{e^\zeta - 1} = \frac{{S_{12}}^2}{M_1(S_2 - M_2)}$, where \(\zeta = M_1(\tilde{\nu} - 1)^2\). The function \(\frac{\zeta}{e^\zeta - 1}\) is well-known in mathematical analysis and appears in the generating function for the Bernoulli numbers \(B_n\), satisfying the expansion $\frac{\zeta}{e^\zeta - 1} = \sum_{n = 0}^\infty \frac{B_n \zeta^n}{n!} = 1 - \frac{\zeta}{2} + \frac{\zeta^2}{12} - \frac{\zeta^4}{720} + \cdots$. This series representation may offer a more computationally efficient approach for approximating \(\tilde{\nu}\) in the regime of small \(\zeta = M_1(\tilde{\nu} - 1)^2\).


\paragraph{Case I ($\beta = 1$ or $\beta = -1$):}
Solving for $\tilde{\gamma}$ (and thus $\tilde{\nu}$), one may either apply Equation~\ref{eq: nu exp}, or alternatively, equate the model-based expression for ${Cov}(X_1, X_2)$ (Equation \ref{eq: Cov(X1X2) exp}) with the sample covariance. Using the latter, this yields:
\begin{equation}
    \tilde{\alpha}(1 - \tilde{\nu})\tilde{\mu}^{\tilde{\nu} - 1} \text{sign}(\beta) = S_{12}. \nonumber
\end{equation}
Applying the Lambert $W$ function allows us to express $\tilde{\nu}$ explicitly as:
\begin{equation}
    \tilde{\nu} = W\left( -\frac{S_{12}}{\tilde{\alpha}} \text{sign}(\beta) \right) + 1. \nonumber
\end{equation}
Naturally, since the Lambert $W$ function is only defined for arguments satisfying $-\frac{S_{12}}{\tilde{\alpha}}\text{sign}(\beta)  =  -\frac{S_{12}}{M_1}\text{sign}(\beta)  \geq -\frac{1}{e}$, a necessary condition for the existence of a real-valued solution is $\frac{M_1}{S_{12}}\text{sign}(\beta)  \geq e$. Furthermore, since $0 < \tilde{\nu} < 1$, we restrict attention to the principal branch $W_0$ of the Lambert $W$ function to ensure that the solution remains within the admissible parameter space, with the additional requirement that $-\frac{S_{12}}{\tilde{\alpha}} \text{sign}(\beta) < 0 \implies \text{sign}(\beta) S_{12} >0$. \\

Furthermore, we solve for solve for $\tilde{\delta}$ just as before, by merely equating Equation \ref{eq: E(X2) exp} to the sample mean:
\begin{eqnarray}
	\tilde{\delta} &=& M_2 - \text{sign}(\beta) (1 - \tilde{\mu}^{\tilde{\nu} - 1}).
	\nonumber
\end{eqnarray} 
\paragraph{Case II ($\gamma = 1$):}
Since $\gamma =1$, we have $\nu = \frac{1}{e}$, hence we equate the model-based expression for ${Cov}(X_1, X_2)$ (Equation \ref{eq: Cov(X1X2) exp}) with the sample covariance to obtain:
\begin{eqnarray}
    \tilde{\beta} = \frac{S_{12}}{\tilde{\alpha}(1- \frac{1}{e})\tilde{\mu}^{( \frac{1}{e} - 1)}}. \nonumber
\end{eqnarray}
Afterwhich, we solve for $\tilde{\delta}$ equating Equation \ref{eq: E(X2) exp} to the sample mean:
\begin{eqnarray}
	\tilde{\delta} &=& M_2 - \tilde{\beta}(1 - \tilde{\mu}^{\frac{1}{e} - 1}).
	\nonumber
\end{eqnarray}
\paragraph{Case IV ($\gamma = 1$ and $\beta = 1$ or $\gamma = 1$ and $\beta = -1$):}
We solve for $\tilde{\delta}$ by merely equating Equation \ref{eq: E(X2) exp} to the sample mean:
\begin{equation}
	\tilde{\delta}  = M_2 - \text{sign}(\beta)(1 - \tilde{\mu}^{\frac{1}{e} - 1}).
	\nonumber
\end{equation}

\paragraph{Case V ($\delta = 0$ and $\gamma = 1$):}
We solve for $\tilde{\beta}$ by merely equating Equation \ref{eq: E(X2) exp} to the sample mean:
\begin{equation}
	\tilde{\beta}  = \frac{M_2}{1 - \tilde{\mu}^{\frac{1}{e}-1}}.
	\nonumber
\end{equation}

\subsection{Correlations}
We analyse the correlation between $X_1$ and $X_2$. Because each model imposes bounds on the attainable correlation $\rho$, we report these feasible ranges explicitly; otherwise users might apply a model to data with $\rho$ outside its admissible set, leading to misspecification or nonexistence of estimators. We utilise the multivariate $\rho$ expression in Equation \ref{eq: rho exp}, that is $\rho(\alpha, \beta, \gamma, \delta)$ noting again $\mu = e^\alpha$ and $\nu = e^{-\nu}$. Now one may simply ascertain $\rho$ bounds through numerical optimisation (and we utilise this as a safety check as well) to find said bounds, yet we use the approach of studying asymptotic regimes - where these limits reveal which directions in the parameter space push $\rho$ toward its attainable extrema. Furthermore, we note that $\lim_{\alpha \to 0^+} \rho(\alpha, \beta, \gamma, \delta) = 0$ (since $X_1$ becomes nearly degenerate and the induced dependence vanishes), hence we do not take this limit first in the subsequent sections as this would not offer us greater insight.

\subsubsection{Full exponential model}
With regards to the full exponential model, we note that by taking the limits of Equation \ref{eq: rho exp}(limits are commutable):
\begin{equation}
    \lim_{\gamma \to \infty}\lim_{\beta\to \infty}\rho(\alpha,\beta, \gamma, \delta)  = \sqrt{\frac{\alpha}{\mu - 1}}, \label{eq: full rho bounds 1}
\end{equation}
and 
\begin{equation}
    \lim_{\gamma \to \infty}\lim_{\beta\to -\infty}\rho(\alpha,\beta, \gamma, \delta)  = -\sqrt{\frac{\alpha}{\mu - 1}}.\label{eq: full rho bounds 2}
\end{equation}
Seeing as these results are not functions of $\delta$, we confirm that the bounds of the correlation function of the full exponential model, $\rho(\alpha,\beta, \gamma, \delta)$, and the bounds of the correlation function of the Case III sub-model, $\rho(\alpha,\beta, \gamma)$, are identical. These bounds being $-1<\rho <1$ by taking $\lim_{\alpha\to 0^+}$ of Equation \ref{eq: full rho bounds 1} and \ref{eq: full rho bounds 2} noting said equations are monotonic with respect to $\alpha$. 
\subsubsection{Case I}
\paragraph{$\beta =1$}
Similarly, we take the limits of Equation \ref{eq: rho exp} after substituting in $\beta = 1$, and noting for positive correlation (that is, $\beta >0$) we must have $\delta \geq 0$, we obtain (limits are commutable):
\begin{equation}
    \lim_{\gamma \to \infty}\lim_{\delta\to 0^+}\rho(\alpha,\beta = 1, \gamma, \delta)  = \sqrt{\frac{\alpha}{\mu^2- 1}}, \label{eq: case1 rho bounds 1}
\end{equation}
which elucidates that the bounds for $\rho$ are $0 < \rho < \frac{\sqrt{2}}{2}$ after taking $\lim_{\alpha\to 0}$ to Equation \ref{eq: case1 rho bounds 1}, noting said equations monotonicity with respect to $\alpha$. 

\paragraph{$\beta = -1$}
Likewise, we take the limits of Equation \ref{eq: rho exp} after substituting in $\beta = -1$, and noting for negative correlation we must have $\delta \geq -\beta$, we obtain (limits are commutable):
\begin{equation}
    \lim_{\gamma \to \infty}\lim_{\delta\to 1^+}\rho(\alpha,\beta = -1, \gamma, \delta)  = -\sqrt{\frac{\alpha}{2\mu- 1}}, \label{eq: case1 neg rho bounds 1}.
\end{equation}
Since Equation \ref{eq: case1 neg rho bounds 1} is not monotonic with respect to $\alpha$, we rather computationally minimise: $\alpha^* = \operatorname*{argmin}_{\alpha} -\sqrt{\frac{\alpha}{2\mu- 1}}   = 0.76804$, such that $-\sqrt{\frac{\alpha^*}{2e^{\alpha^*}- 1}} = -0.48162$. Hence, we obtain the bounds $-0.48162 < \rho < 0$.

\subsection{Case II}
Similarly, taking the limits of Equation \ref{eq: rho exp}:
\begin{equation}
    \lim_{\beta\to \infty}\rho(\alpha,\beta, \gamma = 1, \delta)  =-\frac{\sqrt{\alpha} \frac{1-e}{e}}{\sqrt{e^{\frac{\alpha(1-e)^2}{e^2}}-1}}, \label{eq: case2 rho bounds 1}
\end{equation}
and 
\begin{equation}
    \lim_{\beta\to -\infty}\rho(\alpha,\beta, \gamma = 1, \delta)  = \frac{\sqrt{\alpha} \frac{1-e}{e}}{\sqrt{e^{\frac{\alpha(1-e)^2}{e^2}}-1}}. \label{eq: case2 rho bounds 2}
\end{equation}
Afterwhich, taking $\lim_{\alpha \to 0^+}$ on Equation \ref{eq: case2 rho bounds 1} and \ref{eq: case2 rho bounds 2}, we obtain $-1 < \rho < 1$.

\subsection{Case IV}
\paragraph{$\beta =1$}
Again, taking the limits of Equation \ref{eq: rho exp}, and noting for positive correlation (that is, $\beta >0$) we must have $\delta \geq 0$:
\begin{equation}
    \lim_{\delta\to 0^+}\rho(\alpha,\beta = 1, \gamma = 1, \delta)  = -\frac{\sqrt\alpha \frac{1-e}{e}e^{\frac{\alpha}{e}}}{\sqrt{e^{2\alpha}\left( 1- e^{2\alpha\frac{1-e}{e}}\right) - e^{\frac{\alpha}{e}+\alpha}\left(e^{\alpha\frac{1-e}{e^2} }-1 \right)}} \label{eq: case4 rho bounds 1}
\end{equation}
Noting the monotonicity of Equation \ref{eq: case4 rho bounds 1}, we take $\lim_{\alpha\to 0^+}$ to obtain $\sqrt{\frac{e -1}{2e-1}} \approx 0.62233$. Hence, our bounds are $ 0<\rho < \sqrt{\frac{e -1}{2e-1}} \approx 0.62233$ also illustrated in Figure \ref{fig: rho case 4 pos} displaying $\rho(\alpha,\beta = 1, \gamma = 1, \delta)$.

\paragraph{$\beta =-1$}
Again, taking the limits of Equation \ref{eq: rho exp}, and noting for negative correlation (that is, $\beta <0$) we must have $\delta \geq -\beta$:
\begin{equation}
    \lim_{\delta\to 1^+}\rho(\alpha,\beta = 1, \gamma = 1, \delta)  = \frac{\sqrt{\alpha}\frac{1-e}{e} e^{\frac{\alpha}{e}}}{\sqrt{e^{\frac{\alpha}{e}+\alpha}\left( e^{\alpha \frac{1-e}{e^2}}- e^{\alpha \frac{1-e}{e}} + 1\right)}} \label{eq: case4 neg rho bounds 1}
\end{equation}
Now Equation \ref{eq: case4 neg rho bounds 1} is not monotonic with respect to $\alpha$, so instead of taking limits, we rather minimise computationally: $\alpha^* = \operatorname*{argmin}_{\alpha}  \left\{\lim_{\delta\to 1}\rho(\alpha,\beta = 1, \gamma = 1, \delta) \right\} = 1.427885$, such that $\left\{\lim_{\delta\to 1}\rho(\alpha,\beta = 1, \gamma = 1, \delta) \right\}_{ \rvert \alpha = \alpha^*} \approx -0.419946$. Hence, our bounds are $ \approx  -0.419946<\rho <0$ also illustrated in Figure \ref{fig: rho case 4 neg} displaying $\rho(\alpha,\beta = -1, \gamma = 1, \delta)$.

\begin{figure}[H]
  \centering
  \begin{subfigure}{0.48\textwidth}
    \centering
    \includegraphics[width=\linewidth]{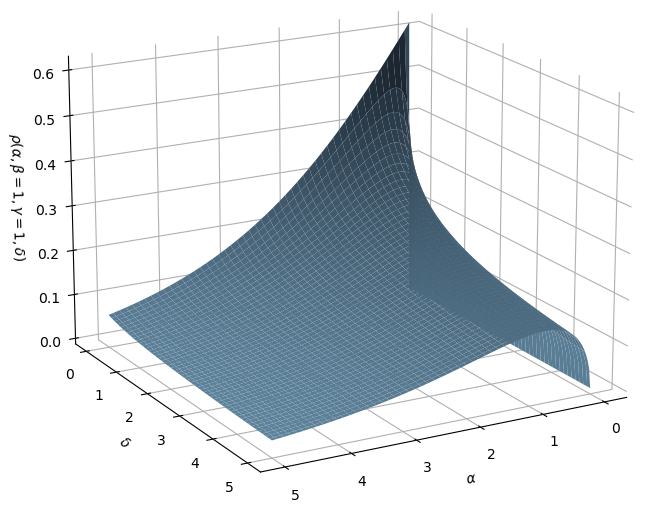}
    \caption{$\rho(\alpha,\beta = 1, \gamma = 1, \delta)$ against $\alpha>0$ and $\delta \geq 0$ noting a maximum achieved at $\approx 0.622$.}
    \label{fig: rho case 4 pos}
  \end{subfigure}
  \hfill
  \begin{subfigure}{0.48\textwidth}
    \centering
    \includegraphics[width=\linewidth]{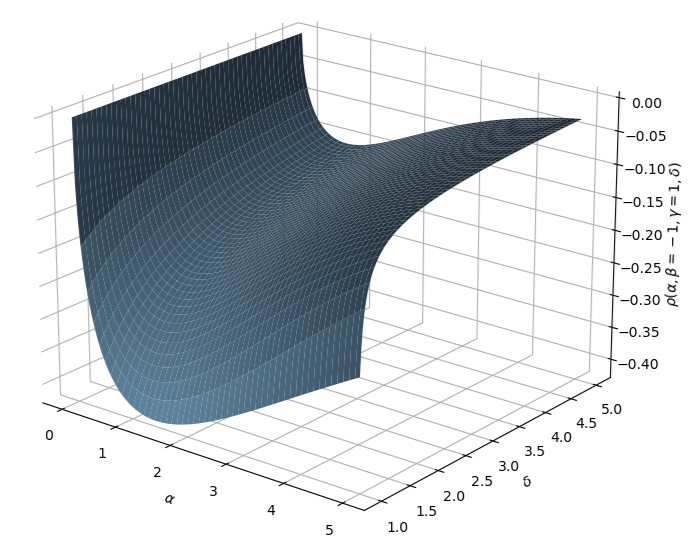}
    \caption{$\rho(\alpha,\beta = -1, \gamma = 1, \delta)$ against $\alpha >0$, $\delta \geq 1$ noting a minimum achieved at $\approx -0.42$.}
    \label{fig: rho case 4 neg}
  \end{subfigure}
  \caption{$\rho(\alpha,\beta = 1, \gamma = 1, \delta)$ and $\rho(\alpha,\beta = -1, \gamma = 1, \delta)$.}
\end{figure}

\subsection{Case V}
Similarly, taking the limits of Equation \ref{eq: rho exp} (limits are commutable):
\begin{equation}
    \lim_{\alpha\to 0^+}\lim_{\beta\to \infty}\rho(\alpha,\beta, \gamma = 1, \delta = 0)  = 1, \nonumber
\end{equation}
and 
\begin{equation}
    \lim_{\alpha\to 0^+}\lim_{\beta\to -\infty}\rho(\alpha,\beta, \gamma = 1, \delta = 0)  = -1.\nonumber
\end{equation}
Hence, the bounds are $-1 < \rho < 1$.

We summarise the correlation $\rho(\alpha, \beta, \gamma, \delta)$ bounds in Table \ref{table: rho bounds}, confirmed with numerical optimisation.
\begin{table}[H]
    \centering
    \begin{tabular}{lcc}
        \hline
        Model & $\rho$ Lower Bound &  $\rho$ Upper Bound\\ \hline
        Full Exponential (including Case III) & -1 & 1 \\ 
        Case I ($\beta = 1$) & 0& $\frac{\sqrt{2}}{2} \approx 0.707$  \\
        Case I ($\beta = -1$) & $\approx-0.482$ &  0 \\ 
        Case II ($\gamma = 1$)& -1 & 1 \\
        Case IV ($\beta = 1, \gamma = 1$)  & 0 & $\sqrt{\frac{e-1}{2e-1}} \approx 0.622$ \\ 
        Case IV ($\beta = -1, \gamma = 1$)  & $\approx -0.420$ & 0 \\
        Case V ($\delta = 0, \gamma = 1$)  & -1 & 1 \\
        \hline
    \end{tabular}
    \caption{$\rho$ lower and upper bounds for the full exponential model and sub-models.}
    \label{table: rho bounds}
\end{table}

\subsection{Maximum likelihood estimation}
Consider again, the data in the form of $\textbf{X}^{(1)}, \textbf{X}^{(2)}, ..., \textbf{X}^{(n)}$ (where for each $i = 1, \ldots, n$, $\textbf{X}^{(i)} = (X_{1i}, X_{2i})^{T} )$ are obtained, which are independent and identically distributed according to Equations \ref{eq: exp Function 1} - \ref{eq: exp Function 2}, the the likelihood function is as follows from Equation \ref{eq:P(X1, X2) exp}, utilising $\nu = e^{-\gamma}$:
\begin{align}
    L(\alpha, \beta, \gamma, \delta) &= \prod_{i=1}^{n}  \frac{e^{-\alpha} \alpha^{x_{1i}}}{x_{1i}!} \frac{e^{-\left(\delta + \beta \left(1 - e^{- \gamma x_{1i}}\right)\right)} \left(\delta + \beta \left(1 - e^{- \gamma x_{1i}}\right)\right)^{x_{2i}}}{x_{2i}!} \nonumber \\
    &= \frac{e^{-n(\alpha +\beta + \delta)} \alpha^{\sum_{i= 1}^nx_{1i}} e^{\beta\sum_{i = 1}^n\nu^{x_{1i}}} \prod_{i = 1}^n  \left(\delta + \beta \left(1 - \nu^ {x_{1i}}\right)\right)^{x_{2i}} }{\prod_{i = 1}^n  \left( x_{1i}!x_{2i}!\right)}. \nonumber
\end{align}
Now the log-likelihood is:
\begin{align}
    l = \log \left(L(\alpha, \beta, \gamma, \delta) \right) = &-n(\alpha +\beta +\delta) + \log(\alpha) \sum_{i =1}^nx_{1i} + \beta \sum_{i = 1}^n\nu^{x_{1i}} + \sum_{i = 1}^n x_{2i}\log \left(\delta + \beta \left(1 - \nu^{x_{1i}}\right) \right) \nonumber \\
    & +h\left(\mathbf{x}_1, \mathbf{x}_2 \right), \label{eq: log lik exp}
\end{align}
where $h\left(\mathbf{x}_1, \mathbf{x}_2 \right) = -\log \left( \prod_{i = 1}^n  \left( x_{1i}!x_{2i}!\right)\right)$. Now since $\frac{\partial l}{\partial \alpha} = 0 \implies \hat{\alpha} = \frac{\sum_{i = 1}^nx_{1i}}{n} = \tilde{\alpha}$, that is, the m.l.e of $\alpha$ is equivalent to the m.m.e of $\alpha$. The m.l.e's of $\beta, \gamma, \delta$ are to be solved numerically using Equation \ref{eq: log lik exp}

\subsection{Likelihood ratio tests}
The general form of the likelihood ratio test is:
\begin{align}
    \Lambda = \frac{sup_{\boldsymbol{\theta}\in\Theta_0}L(\boldsymbol{\theta})}{sup_{\boldsymbol{\theta}\in\Theta}L(\boldsymbol{\theta})} \label{eq: lik ratio test}.
\end{align}
Under regularity conditions, if $H_0$ holds, then 
$-2 \log(\Lambda) \;\;\xrightarrow{d}\;\; \chi^2_{\dim(\Theta) - \dim(\Theta_0)}$. We reject $H_0$ when $-2 \log(\Lambda) > \chi^2_{\dim(\Theta) - \dim(\Theta_0),\,1-\alpha}$
where $\chi^2_{\dim(\Theta) - \dim(\Theta_0),\,1-\alpha}$ denotes the $(1-\alpha)$ critical value of the chi-square distribution with $\dim(\Theta) - \dim(\Theta_0)$ degrees of freedom. We note the natural parameter space of the full exponential model is $\Theta = \{ (\alpha, \beta, \gamma, \delta )^{T}: \alpha > 0, \gamma>  0, \beta \in \mathbb{R} \setminus \{0\}, \delta \geq \max(-\beta, 0) \}$. Let $\hat{\alpha}, \hat{\beta}, \hat{\gamma}, \hat{\delta}$ be the m.l.e's of the full exponential model. Furthermore, we denote ${\hat{\alpha}}^*, \hat{\beta}^*, \hat{\gamma}^*, \hat{\delta}^*$ as any m.l.e under $H_0$ in this section.

\subsubsection{Testing $\beta = \pm 1$}
 Under $H_0$, the natural parameter space is $\Theta_0 = \{ (\alpha, \gamma, \delta )^{T}: \alpha > 0, \gamma > 0, \delta \geq \mathbb{I}\left(\beta <0 \right) \}$. The generalized likelihood ratio test statistic in Equation \ref{eq: lik ratio test} is:
\begin{align}
    \Lambda = \frac{ \frac{e^{-n(\hat{\alpha}^* +\text{sign}(\beta) + \hat{\delta}^*)} {(\hat{\alpha}^{*})}^{\sum_{i= 1}^{n}x_{1i}} e^{\text{sign}(\beta) \sum_{i = 1}^n({\hat{\nu}^*})^{x_{1i}}} \prod_{i = 1}^n  \left(\hat{\delta}^* + \text{sign}(\beta)  \left(1 - ({\hat{\nu}}^{*})^ {x_{1i}}\right)\right)^{x_{2i}} }{\prod_{i = 1}^n  \left( x_{1i}!x_{2i}!\right)}}{\frac{e^{-n(\hat{\alpha} +\hat{\beta} + \hat{\delta})} \hat{\alpha}^{\sum_{i= 1}^{n}x_{1i}} e^{\hat{\beta}\sum_{i = 1}^n\hat{\nu}^{x_{1i}}} \prod_{i = 1}^n  \left(\hat{\delta} + \hat{\beta} \left(1 - \hat{\nu}^ {x_{1i}}\right)\right)^{x_{2i}} }{\prod_{i = 1}^n  \left( x_{1i}!x_{2i}!\right)}}.\nonumber
\end{align}
Since $\hat{\alpha}^* = \hat{\alpha}$, we have:
\begin{eqnarray}
	\Lambda &=& e^{-n\left(\text{sign}(\beta) +\hat{\delta}^* -\hat{\beta} - \hat{\delta}  \right)}e^{\left({\text{sign}(\beta) \sum_{i = 1}^n({\hat{\nu}^*})^{x_{1i}}} -\hat{\beta}\sum_{i = 1}^n\hat{\nu}^{x_{1i}}   \right)} \prod_{i = 1}^{n}  \left[ \frac{\hat{\delta}^* +  \text{sign}(\beta) \left(1 - ({\hat{\nu}}^*)^ {x_{1i}}\right)} {\hat{\delta} + \hat{\beta} \left(1 - \hat{\nu}^ {x_{1i}}\right)} \right]^{x_{2i}}. \nonumber
\end{eqnarray}	
And by taking the logarithm, we obtain:
\begin{align}
	\log{\Lambda} = &-n\left(\text{sign}(\beta) +\hat{\delta}^* -\hat{\beta} - \hat{\delta}  \right) + {\text{sign}(\beta) \sum_{i = 1}^n({\hat{\nu}^*})^{x_{1i}}} -\hat{\beta}\sum_{i = 1}^n\hat{\nu}^{x_{1i}}  \nonumber \\  &+ \sum_{i=1}^{n} x_{2i} \log{ \left[ \frac{\hat{\delta}^* +  \text{sign}(\beta) \left(1 - ({\hat{\nu}}^*)^ {x_{1i}}\right)} {\hat{\delta} + \hat{\beta} \left(1 - \hat{\nu}^ {x_{1i}}\right)} \right]} .\nonumber
\end{align}

\subsubsection{Testing $\gamma =  1$}
 Under $H_0$, the natural parameter space is $\Theta_0 = \{ (\alpha, \beta, \delta )^{T}: \alpha > 0, \beta \in \mathbb{R} \setminus \{0\}, \delta \geq \max(-\beta, 0) \ \}$. The logarithm of the generalized likelihood ratio test statistic in Equation \ref{eq: lik ratio test} is:
 
\begin{align}
	\log{\Lambda} = &-n\left(\hat{\beta}^*+\hat{\delta}^* -\hat{\beta} - \hat{\delta}  \right) + {\hat{\beta}^* \sum_{i = 1}^n\left(\frac{1}{e}\right)^{x_{1i}}} -\hat{\beta}\sum_{i = 1}^n\hat{\nu}^{x_{1i}}  \nonumber + \sum_{i=1}^{n} x_{2i} \log{ \left[ \frac{\hat{\delta}^* +  \hat{\beta}^* \left(1 - \left(\frac{1}{e}\right)^ {x_{1i}}\right)} {\hat{\delta} + \hat{\beta} \left(1 - \hat{\nu}^ {x_{1i}}\right)} \right]}. \nonumber
\end{align}

\subsubsection{Testing $\delta =  0$}
 Under $H_0$, the natural parameter space is $\Theta_0 = \{ (\alpha, \beta, \gamma )^{T}: \alpha > 0, \beta \in \mathbb{R} \setminus \{0\}, \gamma >0 \ \}$. The logarithm of the generalized likelihood ratio test statistic in Equation \ref{eq: lik ratio test} is:
 
\begin{align}
	\log{\Lambda} = &-n\left(\hat{\beta}^*-\hat{\beta} - \hat{\delta}  \right) + {\hat{\beta}^* \sum_{i = 1}^n\left(\hat{\nu}^*\right)^{x_{1i}}} -\hat{\beta}\sum_{i = 1}^n\hat{\nu}^{x_{1i}}  \nonumber + \sum_{i=1}^{n} x_{2i} \log{ \left[ \frac{\hat{\beta}^* \left(1 - \left(\hat{\nu}^*\right)^ {x_{1i}}\right)} {\hat{\delta} + \hat{\beta} \left(1 - \hat{\nu}^ {x_{1i}}\right)} \right]}. \nonumber
\end{align}

\subsubsection{Testing $\beta = \pm 1$ and $\gamma = 1$}
 Under $H_0$, the natural parameter space is $\Theta_0 = \{ (\alpha, \delta )^{T}: \alpha > 0, \delta \geq \mathbb{I}\left(\beta <0 \right)\}$. The logarithm of the generalized likelihood ratio test statistic in Equation \ref{eq: lik ratio test} is:
 
\begin{align}
	\log{\Lambda} = &-n\left(\text{sign}(\beta) +\hat{\delta}^* -\hat{\beta} - \hat{\delta}  \right) + {\text{sign}(\beta) \sum_{i = 1}^n\left(\frac{1}{e}\right)^{x_{1i}}} -\hat{\beta}\sum_{i = 1}^n\hat{\nu}^{x_{1i}}  \nonumber + \sum_{i=1}^{n} x_{2i} \log{ \left[ \frac{\hat{\delta}^* +  \text{sign}(\beta) \left(1 - \left(\frac{1}{e}\right)^ {x_{1i}}\right)} {\hat{\delta} + \hat{\beta} \left(1 - \hat{\nu}^ {x_{1i}}\right)} \right]}.\nonumber
\end{align}

\subsubsection{Testing $\delta =  0$ and $\gamma = 1$}
 Under $H_0$, the natural parameter space is $\Theta_0 = \{ (\alpha, \beta )^{T}: \alpha > 0, \beta \in \mathbb{R} \setminus \{0\}  \}$. The logarithm of the generalized likelihood ratio test statistic in Equation \ref{eq: lik ratio test} is:
 
\begin{align}
	\log{\Lambda} = &-n\left(\hat{\beta}^*-\hat{\beta} - \hat{\delta}  \right) + {\hat{\beta}^* \sum_{i = 1}^n\left(\frac{1}{e}\right)^{x_{1i}}} -\hat{\beta}\sum_{i = 1}^n\hat{\nu}^{x_{1i}}  \nonumber + \sum_{i=1}^{n} x_{2i} \log{ \left[ \frac{\hat{\beta}^* \left(1 - \left(\frac{1}{e}\right)^ {x_{1i}}\right)} {\hat{\delta} + \hat{\beta} \left(1 - \hat{\nu}^ {x_{1i}}\right)} \right]}.\nonumber
\end{align}

\section{Lomax}
We restrict $F(x_1; \boldsymbol{\theta})$, with $\boldsymbol{\theta} = \left[ \gamma', \eta'  \right]^T$, to be the Lomax distribution function. Thus, we have:
\begin{eqnarray}
	X_{1} &\sim & \text{Poisson}(\alpha')  \label{eq: lomax Function 1} \\ 
	X_{2} \mid X_{1} &\sim & \text{Poisson}\left(\delta' +  \beta' \left(1 - \left(\frac{\gamma'}{x_1 +\gamma'}\right)^{\eta'}\right)\right),  \label{eq: lomax Function 2}
\end{eqnarray}
where the scale parameter $\gamma' > 0$ and shape parameter $\eta'>0$ govern the rate of change (gradient) of the regression function. The parameters satisfy $\alpha', \gamma', \eta' > 0$. To ensure that $\lambda_2(x_1) > 0$ for all $x_1 > 0$, it is necessary that $\beta' \in \mathbb{R} \setminus \{0\}$ with $\delta' \geq \max(-\beta', 0)$ - which corresponds to $\delta' \geq -\beta'$ in the case of negative correlation $(\beta' < 0)$ and $\delta' \geq 0$ when the correlation is positive $(\beta' > 0)$.\\

We also examine various sub-models that arise from the Lomax($\eta' = 1$, $\gamma'$) model by constraining specific parameters. These include removing the amplitude control - thereby fixing positive and negative correlation by setting $\beta' = 1$ and $\beta' = -1$ respectively, fixing the rate parameter as $\gamma' = 1$, or jointly constraining both. In addition, we consider a no-intercept model by setting $\delta' = 0$, which by construction precludes the possibility of negative correlation. We further note (as before with the full exponential model), that for the no-intercept model, Case III: $\delta' = 0$, when $x_{1}=0$, the conditional distribution of $x_{2}$ degenerates at zero, implying that $x_{2}=0$ with probability one. Consequently, any empirical observation with $x_{1}=0$ and $x_{2}>0$ is assigned zero probability under the model, which yields a likelihood of zero and indicates model misspecification for such observations. Figure \ref{fig:exp flow} provides a schematic illustration of the resulting sub-models.
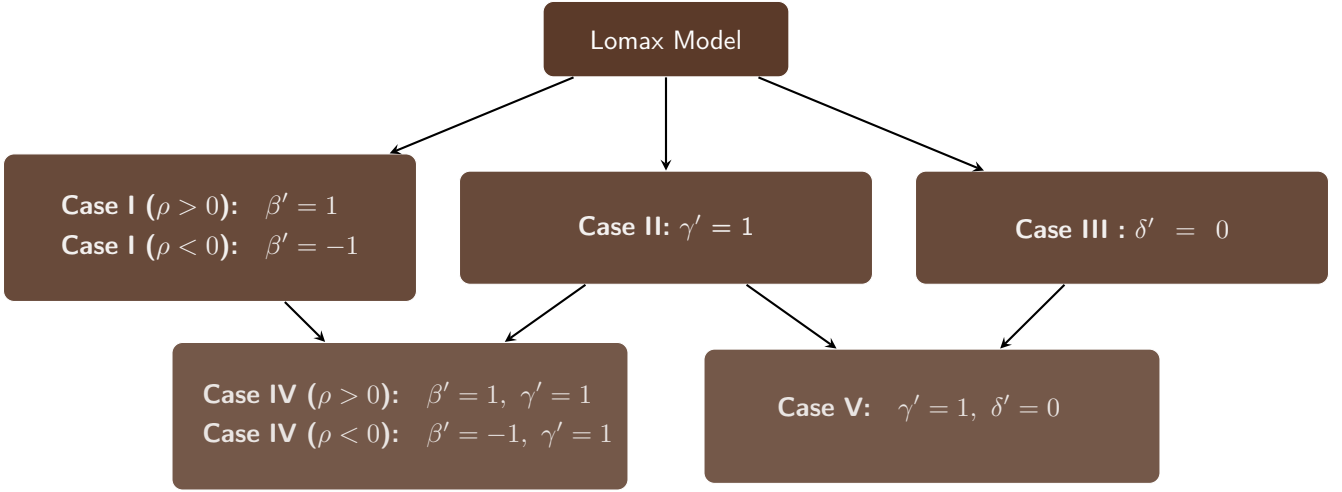
\begin{figure}[H] \centering
	\begin{tikzpicture}[node distance=2cm]
		\node (exp) [io2] {Lomax Model};
		\node (expcase1) [processalt, below of = exp,  xshift = -6cm, yshift = -0.5cm]{\[
        \begin{aligned}
		& \textbf{Case I ($\rho > 0$):} \quad \beta' = 1 \\ 
        & \textbf{Case I ($\rho < 0$):} \quad \beta' = -1 
        \end{aligned}\]};
		\node (expcase2) [processalt, below of = exp, yshift = -0.5cm]{\textbf{Case II:} $\gamma'$ = 1\\ };
	\node (expcase3) [processalt2, below of = expcase1, xshift = 2.5cm, yshift = -0.5cm] {\[
  \begin{aligned}
    &\textbf{Case IV ($\rho > 0$):} \quad \beta' = 1,\ \gamma' = 1 \\
    &\textbf{Case IV ($\rho < 0$):} \quad \beta' = -1,\ \gamma' = 1
  \end{aligned}\]
};
\node (expcase4) [processalt, below of = exp, xshift = 6cm, yshift = -0.5cm]{\textbf{Case III :} $\delta' = 0$};
\node (expcase5) [processalt2, below of = expcase4, xshift = -2.5cm, yshift = -0.5cm] {\[
  \begin{aligned}
    &\textbf{Case V:} \quad \gamma' = 1,\ \delta' = 0 
\end{aligned}\]
};
		\draw [arrow] (exp) -- (expcase1);
		\draw [arrow] (exp) -- (expcase2);
		\draw [arrow] (expcase1) -- (expcase3);
		\draw [arrow] (expcase2) -- (expcase3);
		\draw [arrow] (exp) -- (expcase4);
		\draw [arrow] (expcase4) -- (expcase5);
		\draw [arrow] (expcase2) -- (expcase5);
	\end{tikzpicture}
	\caption{Flow diagram of Lomax sub-models.}
    \label{fig:exp flow}
\end{figure}

Before proceeding, we note that, for the exponential model, taking the limit $\lim_{\gamma \to \infty}$ of the regression form in Equation \ref{eq: exp Function 2} yields the same result as taking $\lim_{\gamma' \to 0^+}$ of the regression form in Equation \ref{eq: lomax Function 2}. Conversely, taking the limit $\lim_{\gamma \to 0^+}$ in Equation \ref{eq: exp Function 2} corresponds to $\lim_{\gamma' \to \infty}$ in Equation \ref{eq: lomax Function 2}. We summarize these equivalences below:
\begin{align}
    \lim_{\gamma \to \infty} \left\{\delta +\beta\left(1 - e^{-\gamma x_1} \right)\right\} & =\delta +\beta, \nonumber\\
    \lim_{\gamma' \to 0^+} \left\{\delta' +  \beta' \left(1 - \left(\frac{\gamma'}{x_1 +\gamma'}\right)^{\eta'}\right)\right\} \nonumber & = {\delta}' + {\beta}', \nonumber
\end{align}
and 
\begin{align}
    \lim_{\gamma \to 0} \left\{\delta +\beta\left(1 - e^{-\gamma x_1} \right)\right\} & =\delta, \nonumber\\
    \lim_{\gamma' \to \infty} \left\{\delta' +  \beta' \left(1 - \left(\frac{\gamma'}{x_1 +\gamma'}\right)^{\eta'}\right)\right\} \nonumber & = {\delta}'. \nonumber
\end{align}
This implies that under these limits, the exponential and Lomax models become equivalent, with $\alpha = \alpha', \beta = \beta'$ and $\delta = \delta'$ (where $\eta'$ becomes inconsequential under these limits). Specifically, we have:
$$
X_1 \sim \text{Poisson}(\alpha), \quad X_2 \mid X_1 \sim \text{Poisson}(\delta + \beta) \quad \text{as } \lim_{\gamma \to \infty} \text{ or } \lim_{\gamma' \to 0^+},
$$  
and  
$$
X_1 \sim \text{Poisson}(\alpha), \quad X_2 \mid X_1 \sim \text{Poisson}(\delta) \quad \text{as } \lim_{\gamma \to 0} \text{ or } \lim_{\gamma' \to \infty}.
$$
That is to say, both exponential and Lomax models become independence models - where $X_2$ is no longer conditioned on $X_1$.
\subsection{Probability mass functions}
Following from the definition of a Poisson distribution:
\begin{eqnarray}
	P(X_{1}) &=& \frac{e^{-\alpha'}(\alpha')^{x_{1}}}{x_{1}!}, \nonumber\\
	P(X_{2}|X_{1}) &=& \frac{e^{-\delta' - \beta' \left(1 -  \left(\frac{\gamma'}{x_1 +\gamma'}\right)^{\eta'}   \right)} \left\{\delta' + \beta' \left(1 - \left(\frac{\gamma'}{x_1 +\gamma'}\right)^{\eta'}\right) \right\}^{x_2}       } {x_{2}!}  ,\nonumber \\    
	P(X_{1}, X_{2}) &=& P(X_{1})P(X_{2}|X_{1}) \nonumber\\  
	&=& \frac{e^{-\alpha'}(\alpha')^{x_{1}}}{x_{1}!} \frac{e^{-\delta' - \beta' \left(1 -  \left(\frac{\gamma'}{x_1 +\gamma'}\right)^{\eta'}   \right)} \left\{\delta' + \beta' \left(1 - \left(\frac{\gamma'}{x_1 +\gamma'}\right)^{\eta'}\right) \right\}^{x_2}       } {x_{2}!}, \label{eq:P(X1, X2) lomax}
\end{eqnarray}
where $x_{1}$, $x_{2}$ $\in$ $\{0, 1, 2...\}$, $\alpha', \gamma', \eta' > 0$ and $\beta' \in \mathbb{R}\setminus \{0\}$ with $\delta' \geq \max(-\beta', 0)$.
\subsection{Moments} \label{sec: moments lomax}
We again denote $e^{\alpha'} = \mu'$, and for ease of simplification, we express the summations in terms of hypergeometric functions, and accordingly restrict $\eta'$ to take discrete values; that is, we assume $\eta' \in \mathbb{N}$. Now:
\begin{align}
    E(X_1) &= \alpha', \label{eq: E(X1) lomax} \\
    E(X_2) = E_{X_{1}}\{E(X_2|X_1 = x_1)\}  &=  \delta' +  \beta' - \beta' E_{X_1} \left\{\left(\frac{\gamma'}{x_1 +\gamma'}\right)^{\eta'}\right\}\nonumber \\
    & = \delta'+ \beta'\left(1 - \frac{1}{\mu'}{}_{\eta'}F_{\eta'}(\gamma', \ldots, \gamma', (\gamma'+1), \ldots, (\gamma'+1); \alpha') \right), \label{eq: E(X2) lomax}
\end{align}
since:
\begin{align}
    E_{X_1} \left\{\left(\frac{\gamma'}{x_1 +\gamma'}\right)^{\eta'}\right\} &= \sum_{i = 0}^\infty \left(\frac{\gamma'}{\gamma'+i} \right)^{\eta'} \frac{e^{-\alpha'}(\alpha')^i}{i!} \nonumber\\
    &= \sum_{i = 0}^\infty \left(\frac{(\gamma')_i}{(\gamma'+1)_i} \right)^{\eta'} \frac{e^{-\alpha'}(\alpha')^i}{i!} \nonumber\\
    & = \frac{1}{e^{\alpha'}} {}_{\eta'}F_{\eta'}(\gamma', \ldots, \gamma', (\gamma'+1), \ldots, (\gamma'+1); \alpha'), \label{eq: E(e) lomax}
\end{align}
and using the Pochhammer symbol $
(q)_n = 
\begin{cases}
1 & \text{if } n = 0 \\
q(q+1)\cdots(q+n-1) & \text{if } n > 0,
\end{cases}
$
\begin{align}
    \frac{\gamma'}{\gamma'+i} & = \frac{\gamma'(\gamma'+1)(\gamma'+2)\cdots(\gamma'+i - 2)(\gamma'+i-1)}{(\gamma'+1)(\gamma'+2)(\gamma'+3)\cdots(\gamma'+i - 1)(\gamma'+i)} \nonumber \\
    & = \frac{(\gamma')_i}{(\gamma'+1)_i}. \nonumber
\end{align}
We denote $_{\eta'}F_{\eta'}(\gamma', \ldots, \gamma', (\gamma'+1), \ldots, (\gamma'+1); \alpha')$ as $_{\eta'}F_{\eta'}(\gamma')$ going forward. Now, 
\begin{align}
    Var(X_1) &= \alpha' \label{eq: Var(X1) lomax},\\
    Var_{X_2\mid X_1}(X_2\mid X_1) &=  \delta' +  \beta' \left(1 - \left(\frac{\gamma'}{x_1 +\gamma'}\right)^{\eta'}\right) \nonumber \\
    &= E_{X_2\mid X_1}(X_2\mid X_1). \label{eq: Var(X2|X_1) lomax}
\end{align}
Using Equations \ref{eq: E(X2) lomax}, \ref{eq: E(e) lomax} and \ref{eq: Var(X2|X_1) lomax}, we have:
\begin{align}
    Var(X_2) &= E_{X_1} \{Var_{X_2|X_1}(X_2|X_1) \} + Var_{X_1}\{E_{X_2|X_1}(X_2|X_1)    \} \nonumber \\
	&= E(X_2) + (\beta')^{2} Var_{X_1} \left\{\left( \frac{\gamma'}{x_1 +\gamma'}\right)^{\eta'} \right\} \nonumber \\
    & = E(X_2) + (\beta')^2 \left(E_{X_1} \left\{\left( \frac{\gamma'}{x_1 +\gamma'}\right)^{2\eta'} \right\} -  \left(E_{X_1}\left\{\left( \frac{\gamma'}{x_1 +\gamma'}\right)^{\eta'} \right\} \right)^2 \right) \nonumber \\
    & = \delta' + \beta'\left(1 -\frac{1}{\mu'} {}_{\eta'}F_{\eta'}(\gamma')  \right) + (\beta')^2\left(\frac{1}{\mu'} {}_{2\eta'}F_{2\eta'}(\gamma') - \frac{1}{(\mu')^2}\left({}_{\eta'}F_{\eta'}(\gamma')\right)^2 \right). \label{eq: Var(X2) lomax}
\end{align}
 Furthermore from Equation \ref{eq: E(X1) lomax}:
\begin{align}
	E(X_1 X_2) = E_{X_1} \{ E(X_1 X_2 | X_1 )  \}  &= E_{X_1} \{ X_1 E( X_2 | X_1 )  \} \nonumber \\
	&= E_{X_1} \left\{ X_1\left(\delta' +  \beta' \left(1 - \left(\frac{\gamma'}{x_1 +\gamma'}\right)^{\eta'}\right) \right) \right\} \nonumber \\
	&= \alpha' (\delta' +  \beta'  ) - \beta' E_{X_1} \left\{ X_1 \left(\frac{\gamma'}{x_1 +\gamma'}\right)^{\eta'} \right\}  \nonumber \\
    & = \alpha'\left(\delta' + \beta'\left( 1 - \frac{1}{\mu'}\left(\frac{\gamma'}{\gamma'+1}\right)^{\eta'}{}_{\eta'}F_{\eta'}(\gamma'+1) \right) \right). \label{eq: E(X1X2) lomax} 
\end{align}
since: 
\begin{align}
    E_{X_1} \left\{ X_1 \left(\frac{\gamma'}{x_1 +\gamma'}\right)^{\eta'} \right\}  & = \sum_{i = 0}^\infty i\left(\frac{\gamma'}{\gamma'+i} \right)^{\eta'} \frac{e^{-\alpha'}(\alpha')^i}{i!} \nonumber\\
    & =  \sum_{i = 1}^\infty \left(\frac{\gamma'}{\gamma'+(i-1)+1} \right)^{\eta'} \frac{e^{-\alpha'}(\alpha')^{i-1}\alpha'}{(i-1)!} \nonumber\\
    & = \alpha' \frac{1}{e^{\alpha'}} \left(\frac{\gamma'}{\gamma'+1}\right)^{\eta'}\sum_{j = 0}^\infty \left(\frac{\gamma'+1}{\gamma'+j+1} \right)^{\eta'} \frac{(\alpha')^j}{j!} \nonumber\\
    & =\alpha' \frac{1}{e^{\alpha'}} \left(\frac{\gamma'}{\gamma'+1}\right)^{\eta'}\sum_{j = 0}^\infty \left(\frac{(\gamma'+1)_j}{(\gamma'+2)_j} \right)^{\eta'} \frac{(\alpha')^{j}}{j!} \nonumber\\
    & = \alpha' \frac{1}{e^{\alpha'}} \left(\frac{\gamma'}{\gamma'+1}\right)^{\eta'}{}_{\eta'}F_{\eta'}\left((\gamma'+1), \ldots, (\gamma'+1),(\gamma'+2),\ldots, (\gamma'+2); \alpha'\right). \nonumber
\end{align}
and using:
\begin{align}
    \frac{\gamma'+1}{\gamma'+j+1} & = \frac{(\gamma'+1)(\gamma'+2)(\gamma'+3)\cdots(\gamma'+1+j - 2)(\gamma'+1+j-1)}{(\gamma'+2)(\gamma'+3)(\gamma'+4)\cdots(\gamma'+2+j - 2)(\gamma'+2 +j -1 )} \nonumber \\
    & = \frac{(\gamma'+1)_j}{(\gamma'+2)_j}. \nonumber
\end{align}
Now from Equations \ref{eq: E(X1) lomax}, \ref{eq: E(X2) lomax} and \ref{eq: E(X1X2) lomax}, the covariance between $X_1$ and $X_2$ is:
\begin{align}
	Cov(X_1, X_2) &= E(X_1 X_2) - E(X_1)E(X_2) \nonumber \\
	&= \frac{\alpha'\beta'}{\mu'}\left( {}_{\eta'}F_{\eta'}(\gamma') - \left(\frac{\gamma'}{\gamma'+1}\right)^{\eta'}{}_{\eta'}F_{\eta'}(\gamma'+1) \right).\label{eq: Cov(X1X2) lomax}
\end{align}
We note that since ${}_{\eta'}F_{\eta'}(\gamma')  = \sum_{i = 1}^\infty\left(\frac{\gamma'}{\gamma'+i} \right)^{\eta'} \frac{\alpha'}{i!}$ and $ \left(\frac{\gamma'}{\gamma'+1}\right)^{\eta'}{}_{\eta'}F_{\eta'}(\gamma'+1) = \sum_{i = 1}^\infty\left(\frac{\gamma'}{\gamma'+1+i} \right)^{\eta'} \frac{\alpha'}{i!}$, we confirm that ${}_{\eta'}F_{\eta'}(\gamma') > \left(\frac{\gamma'}{\gamma'+1}\right)^{\eta'}{}_{\eta'}F_{\eta'}(\gamma'+1)$ implying that the sign of $Cov(X_1, X_2)$ is solely dictated by the sign of $\beta'$ as postulated before. Now the associated correlation between $X_1$ and $X_2$, using Equations \ref{eq: Var(X1) lomax}, \ref{eq: Var(X2) lomax} and \ref{eq: Cov(X1X2) lomax}, is:
\begin{align}
	\rho(\alpha', \beta', \gamma', \delta', \eta') &= \frac {Cov(X_1 X_2)}{\sqrt{Var(X_1) Var(X_2)}} \nonumber \\
&= \frac{\frac{ \sqrt{\alpha'} \beta'}{\mu'} \left( {}_{\eta'}F_{\eta'}(\gamma') - \left(\frac{\gamma'}{\gamma'+1}\right)^{\eta'}{}_{\eta'}F_{\eta'}(\gamma'+1) \right)}{\sqrt{(\beta')^2\left(\frac{1}{\mu'} {}_{2\eta'}F_{2\eta'}(\gamma') - \frac{1}{(\mu')^2}\left({}_{\eta'}F_{\eta'}(\gamma')\right)^2 \right) + \beta'\left(1 -\frac{1}{\mu'} {}_{\eta'}F_{\eta'}(\gamma')  \right) + \delta'}}.\label{eq: rho lomax}
\end{align}
\subsection{Method of moments estimation}
We use the moments derived for the Lomax model in Section \ref{sec: moments lomax} to ascertain m.m.e's for the associated parameters. Nevertheless, the limited number of independent moments renders this approach inadequate for estimating $\tilde{\eta}'$. Being such, we derive moments for a Lomax$\left(\eta' = 1, \gamma' \right)$ model instead.
\subsubsection{Approach I}
Noting the similarities in the formulations between Equation \ref{eq: E(X1) exp} in Section \ref{sec: moments exp} and Equation \ref{eq: E(X1) lomax} in Section \ref{sec: moments lomax}, we may conclude that the m.m.e of $\alpha$ is the same as the m.m.e of $\alpha'$ - that is $\tilde{\alpha} = \tilde{\alpha}'$. Likewise, the similarities between Equations \ref{eq: E(X2) exp}, \ref{eq: Var(X2)} and \ref{eq: Cov(X1X2) exp} in Section \ref{sec: moments exp} and Equations \ref{eq: E(X2) lomax}, \ref{eq: Var(X2) lomax} and \ref{eq: Cov(X1X2) lomax} in Section \ref{sec: moments lomax} elucidate the equivalency of moments:
\begin{align}
    \tilde{\delta} + \tilde{\beta}\left(1 - \frac{1}{\tilde{\mu}}\tilde{\mu}^{\tilde{\nu}}\right) & =  \tilde{\delta'} + \tilde{\beta'}\left(1 - \frac{1}{\tilde{\mu'}} {}_{1}F_{1}(\tilde{\gamma}')\right),
 \label{eq: moment equiv 1 hyper} \\    
    \tilde{\beta}^2 \left(\frac{1}{\tilde{\mu}}\tilde{\mu}^{\tilde{\nu}^2} -\frac{1}{(\tilde{\mu})^2}\tilde{\mu}^{2\tilde{\nu}}\right) & = (\tilde{\beta}')^2\left(\frac{1}{\tilde{\mu}'} {}_{2}F_{2}(\tilde{\gamma}') - \frac{1}{(\tilde{\mu}')^2}\left({}_{1}F_{1}(\tilde{\gamma}')\right)^2 \right),\label{eq: moment equiv 2 hyper}
\end{align}
and
\begin{align}
        \tilde{\beta}(1 - \tilde{\nu})\tilde{\mu}^{\tilde{\nu}} & = \tilde{\beta'}\left({}_{1}F_{1}(\tilde{\gamma}') - \left( \frac{\tilde{\gamma'}}{ \tilde{\gamma'}+1 }\right){}_{1}F_{1}(\tilde{\gamma}'+1) \right). \label{eq: moment equiv 3 hyper}
\end{align}
Subsequently, we may relax the discreteness assumption on $\eta'$ - reverting back to $\eta' >0$ - and substitute ${}_{\eta'}F_{\eta'}(\gamma')  = \sum_{i = 1}^\infty\left(\frac{\gamma'}{\gamma'+i} \right)^{\eta'} \frac{\alpha'}{i!}$ and $ \left(\frac{\gamma'}{\gamma'+1}\right)^{\eta'}{}_{\eta'}F_{\eta'}(\gamma'+1) = \sum_{i = 1}^\infty\left(\frac{\gamma'}{\gamma'+1+i} \right)^{\eta'} \frac{\alpha'}{i!}$ into Equations \ref{eq: moment equiv 1 hyper}, \ref{eq: moment equiv 2 hyper} and \ref{eq: moment equiv 3 hyper}. \\

Equations \ref{eq: moment equiv 1 hyper}, \ref{eq: moment equiv 2 hyper} and \ref{eq: moment equiv 3 hyper} elucidate that if m.m.e's for $\alpha$ (hence $\mu$), $\beta$, $\gamma$ (hence $\nu$) and $\delta$ are calculated for the exponential model using Equations \ref{eq: mme alpha exp}, \ref{eq: mme beta exp}, \ref{eq: nu exp} and \ref{eq: mme delta exp}, one may use Equations \ref{eq: moment equiv 1 hyper}, \ref{eq: moment equiv 2 hyper} and \ref{eq: moment equiv 3 hyper} to numerically determine m.m.e's of $\alpha'$, $\beta'$, $\gamma'$ and $\delta'$ of the Lomax$\left(\eta' = 1, \gamma' \right)$ model.
\subsubsection{Approach II}
Equating Equation \ref{eq: E(X1) lomax} to its sample mean:
\begin{eqnarray}
	\tilde{\alpha}' &=& M_{1} \label{eq: mme alpha lomax}, \\
	e^{\tilde{\alpha}'} &=& e^{M_{1}}  = \tilde{\mu}',\nonumber
\end{eqnarray}	
where Equation \ref{eq: mme alpha lomax} is the applicable m.m.e for all Lomax models. Consider the remaining moments, $\tilde{\beta}'$, $\tilde{\gamma}'$ and $\tilde{\delta}'$ for each model separately:
\paragraph{Lomax($\eta' = 1$, $\gamma'$) model:}
Taking the quotient of Equation \ref{eq: Cov(X1X2) lomax} squared and the difference between Equation \ref{eq: Var(X2) lomax} and Equation \ref{eq: E(X2) lomax}, and equating the result to the solution obtained from identical operations on the associated sample moments, an implicit equation w.r.t $\tilde{\gamma}'$ is obtained, to be solved numerically:
\begin{eqnarray}
\frac{(\tilde{\alpha}')^2\left[{}_{1}F_{1}(\tilde{\gamma}') - \left(\frac{\tilde{\gamma}'}{\tilde{\gamma}'+1}\right){}_{1}F_{1}(\tilde{\gamma}'+1)  \right]^2}{\tilde{\mu}' {}_{2}F_{2}(\tilde{\gamma}') -  [{}_{1}F_{1}(\tilde{\gamma}')]^2 } & = \frac{{S_{12}}^2}{S_2 - M_2}, \label{eq: gamma lomax}
\end{eqnarray}
and to ensure the expression remains real-valued we obtain the condition:
\begin{align}
      S_2 &\neq M_2.\nonumber
\end{align}
Since $\forall x_1, x_2 \in {0, 1, 2, \dots}$ and $\tilde{\gamma}' > 0$, it is necessary to establish bounds for the quantity $\frac{{S_{12}}^2}{S_2 - M_2}$ to ensure that the condition $\tilde{\gamma}'> 0$ is satisfied.\\

Now, since the left-hand side of Equation \ref{eq: gamma lomax} is strictly increasing in $\tilde{\gamma}'>0$ (as proven in Theorem \ref{thm: f(gamma') strict inc}), we may take the limits $\lim_{\tilde{\gamma}' \to 0^+}$ and $\lim_{\tilde{\gamma}' \to \infty}$ of the left-hand side of Equation \ref{eq: gamma lomax} to obtain its lower and upper bounds, respectively. Furthermore, by substituting using Equations \ref{eq: moment equiv 1 hyper}, \ref{eq: moment equiv 2 hyper}, and \ref{eq: moment equiv 3 hyper} into the left-hand side of Equation \ref{eq: gamma lomax}, we recover the left-hand side of Equation \ref{eq: nu exp}. Since we already established that taking the limits $\lim_{\tilde{\gamma} \to \infty}$ and $\lim_{\tilde{\gamma} \to 0^+}$ in the exponential model is equivalent to taking $\lim_{\tilde{\gamma}' \to 0^+}$ and $\lim_{\tilde{\gamma}' \to \infty}$ in the Lomax model, it is evident that we obtain the same lower and upper bounds for the left-hand side of Equation \ref{eq: nu exp}, as given in Equation \ref{eq: full bounds 2} (we show $\lim_{\gamma' \to 0^+}$ of the left-side expression in Equation \ref{eq: gamma lomax} in in Theorem \ref{thm: limit gamma lomax} regardless). Thus, the bounds for the quantity $\frac{S_{12}^2}{S_2 - M_2}$ are as follows (equivalent to Equation \ref{eq: full bounds 2}):

\begin{eqnarray}
	\frac{M_1^2}{e^{M_1} - 1}&<& \frac{{S_{12}}^2}{{S_2 - M_2}}   <  M_1 \label{eq: full bounds 2 lomax}.
\end{eqnarray}
Figure \ref{fig: gamma mme lomax} illustrates $\tilde{\gamma}'$ as a function of $\tilde{\alpha} =M_1$ and $\frac{S_{12}^2}{S_2 - M_2}$. Importantly, due to the bounds imposed by $M_1$ on $\frac{S_{12}^2}{S_2 - M_2}$ in Equation \ref{eq: full bounds 2 lomax}, certain combinations of $M_1$, $S_{2}$, $M_{2}$ and $S_{12}$ values may lead to the non-existence of $\tilde{\gamma}'$. This behavior is also evident from the patterns observed in Figure \ref{fig: gamma mme lomax}. 
\begin{figure}[H]
    \centering
    \includegraphics[width=0.5\linewidth]{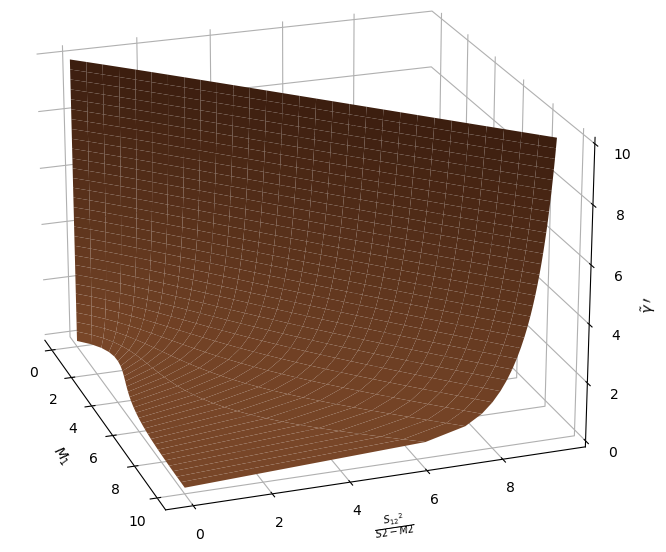}
    \caption{$\tilde{\gamma}'$ vs $M_1$ and $\frac{S_{12}^2}{S_2 - M_2}$ as per Equation \ref{eq: gamma lomax}.}
    \label{fig: gamma mme lomax}
\end{figure}
We then solve for $\tilde{\beta}'$ by equating Equation \ref{eq: Cov(X1X2) lomax} to the sample covariance to obtain:
\begin{equation}
     \tilde{\beta}' = \frac{S_{12}}{ \frac{\tilde{\alpha}'}{\tilde{\mu}'}\left( {}_{1}F_{1}(\tilde{\gamma}') - \left(\frac{\tilde{\gamma}'}{\tilde{\gamma}'+1}\right) {}_{1}F_{1}(\tilde{\gamma}'+1) \right)}.\nonumber
\end{equation}
Finally, solve for $\tilde{\delta}'$ by merely equating Equation \ref{eq: E(X2) lomax} to the sample mean of $X_2$:
\begin{eqnarray}
	\tilde{\delta}' &=& M_2 - \tilde{\beta}'\left(1 -  \frac{1}{\tilde{\mu}'}{}_{1}F_{1}(\tilde{\gamma}')\right).
	\nonumber
\end{eqnarray}
We note that the m.m.e's of Case 3: $\delta' = 0$, may be derived analogously to the approach previously outlined, with the natural exception of excluding the estimator for $\delta'$. 

\paragraph{Case I ($\beta' = 1$ or $\beta' = -1$):}
With regards to solving for $\tilde{\gamma}'$, one may either solve Equation~\ref{eq: gamma lomax}, or alternatively, equate the model-based expression for ${Cov}(X_1, X_2)$ (Equation \ref{eq: Cov(X1X2) exp}) with the sample covariance. This yields:
\begin{equation}
    \text{sign}(\beta')\frac{\tilde{\alpha}'}{\tilde{\mu}'} \left({}_{1}F_{1}(\tilde{\gamma}') -\left(\frac{\tilde{\gamma}'}{\tilde{\gamma}'+1}\right){}_{1}F_{1}(\tilde{\gamma}'+1) \right) = S_{12}, \label{eq: mme gamma case1 lomax}
\end{equation}
or equivalently:
\begin{equation}
    \text{sign}(\beta')\frac{\tilde{\alpha}'}{\tilde{\mu}'} \sum_{i = 0}^\infty\frac{\tilde{\gamma}'}{(\tilde{\gamma}'+i)(\tilde{\gamma}'+1+i)} \frac{(\tilde{\alpha}')^i}{i!}= S_{12}. \nonumber
\end{equation}
Since $\forall x_{1}, x_{2} \in \{0, 1, 2...\}$,  $\tilde{\gamma}' > 0$, it follows that bounds must be established for the quantity $S_{12}$ in order to guarantee the condition $\tilde{\gamma}' > 0$ is satisfied. Since the left-hand side expression of Equation \ref{eq: mme gamma case1 lomax} is not monotonic with respect to $\tilde{\gamma}'$, it is not possible to determine bounds for $S_{12}$ by using limits to guarantee $\tilde{\gamma}'>0$, as was done previously. Figure \ref{fig: mme gamma lomax case 1} illustrates the relationship between $\tilde{\gamma}'$, $\tilde{\alpha}' = M_1$, and $\text{sign}(\beta')S_{12}$. Importantly, due to the inability to establish definitive bounds on $\text{sign}(\beta')S_{12}$ ensuring a positive m.m.e. for $\gamma'$, certain combinations of $M_1$ and $S_{12}$ values may lead to the non-existence of $\tilde{\gamma}'$. This behavior is also evident from the patterns observed in Figure \ref{fig: mme gamma lomax case 1}.
\begin{figure}[H]
    \centering
    \includegraphics[width=0.5\linewidth]{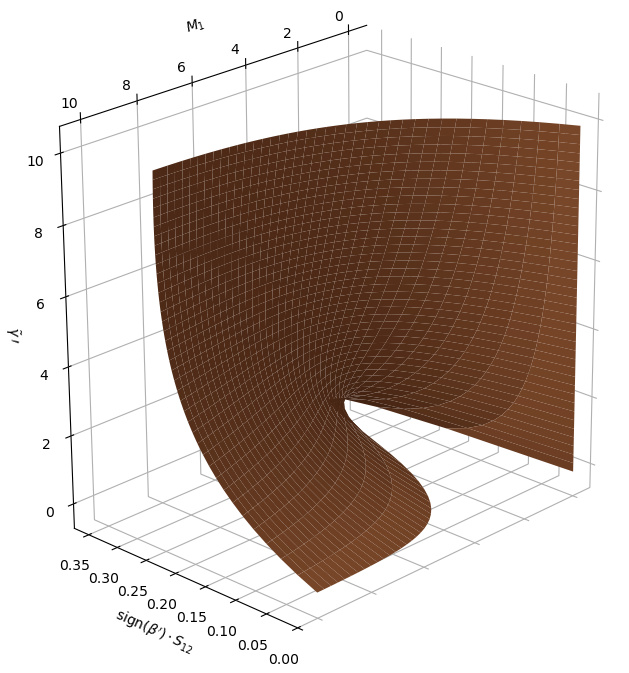}
    \caption{$\tilde{\gamma}'$ vs $M_1$ and $\text{sign}(\beta')$ as per Equation \ref{eq: mme gamma case1 lomax}.}
    \label{fig: mme gamma lomax case 1}
\end{figure}
 
Furthermore, we solve for solve for $\tilde{\delta}'$ just as before, by merely equating Equation \ref{eq: E(X2) lomax} to the sample mean of $X_2$:
\begin{eqnarray}
	\tilde{\delta}' &=& M_2 - \text{sign}(\beta') \left(1 - \frac{1}{\tilde{\mu}' }{}_{1}F_{1}(\tilde{\gamma}')\right).
	\nonumber
\end{eqnarray} 
\paragraph{Case II ($\gamma' = 1$):}
Since $\gamma' =1$, we have ${}_{1}F_{1}(\tilde{\gamma}' = 1) = \frac{1}{\tilde{\alpha}'}\left(\tilde{\mu}' - 1 \right)$ and ${}_{1}F_{1}(\tilde{\gamma}+1 =2) = \frac{1}{(\tilde{\alpha}')^2}\left(\tilde{\mu}'(\tilde{\alpha}'-1)+1\right)$. Now we equate the model-based expression for ${Cov}(X_1, X_2)$ (Equation \ref{eq: Cov(X1X2) lomax}) with the sample covariance to obtain:
\begin{eqnarray}
    \tilde{\beta}' = \frac{S_{12}} {\frac{1}{\tilde{\alpha}'\tilde{\mu}'}\left( \tilde{\mu}' - \tilde{\alpha}' -1\right)}.
\end{eqnarray}
Afterwhich, we solve for $\tilde{\delta}'$ equating Equation \ref{eq: E(X2) lomax} to the sample mean of $X_2$:
\begin{eqnarray}
	\tilde{\delta}' &=& M_2 - \tilde{\beta}'\left(1 - \frac{1}{\tilde{\alpha}'}\left(1 - \frac{1}{\tilde{\mu}'}  \right)\right) .
	\nonumber
\end{eqnarray}
\paragraph{Case IV ($\gamma' = 1$ and $\beta' = 1$ or $\gamma' = 1$ and $\beta' = -1$):}
We solve for $\tilde{\delta}'$ by merely equating Equation \ref{eq: E(X2) lomax} to the sample mean of $X_2$:
\begin{equation}
	\tilde{\delta}'  = M_2 - \text{sign}(\beta')\left(1 - \frac{1}{\tilde{\alpha}'}\left(1 - \frac{1}{\tilde{\mu}'}  \right)\right).
	\nonumber
\end{equation}
\paragraph{Case V ($\delta' = 0$ and $\gamma' = 1$):}
We solve for $\tilde{\beta}'$ by merely equating Equation \ref{eq: E(X2) lomax} to the sample mean of $X_2$:
\begin{equation}
	\tilde{\beta}'  = \frac{M_2}{1 - \frac{1}{\tilde{\mu}'}{}_{1}F_{1}(\tilde{\gamma}')}. \nonumber
	\nonumber
\end{equation}

\subsection{Correlations}
We note, as previously explained, the equivalence of the Lomax model with the exponential model, specifically that $\alpha = \alpha'$, $\beta = \beta'$, and $\delta = \delta'$ under the limit $\lim_{\gamma' \to 0^{+}}$, which corresponds to taking the limit $\lim_{\gamma \to \infty}$ in the exponential model. Consequently, the correlation bounds, for the full Lomax model, the Lomax($\eta' = 1$, $\gamma'$) model and its sub-models, Case I and Case III, coincide with those of their exponential model counterparts. Referring to Equations~\ref{eq: full rho bounds 1} and \ref{eq: full rho bounds 2} for the full exponential model, and Equations~\ref{eq: case1 rho bounds 1} and \ref{eq: case1 neg rho bounds 1} for Case I, we observe that the limit $\lim_{\gamma \to \infty}$ applied to the correlation function of the full exponential model in Equation~\ref{eq: rho exp} directly establishes the equivalence between the exponential and Lomax models. However, this equivalence does not extend to Cases II, IV and V of the exponential model, and therefore, the correlation bounds for the Lomax sub-models Case II, Case IV and Case V require separate derivation. Furthermore, we utilise the Lomax correlation function $ \rho\left(\alpha', \beta', \gamma', \delta', \eta' = 1 \right)$ given in Equation \ref{eq: rho lomax} in this section.

\subsection{Case II}
Since $\gamma' =1$, we have ${}_{1}F_{1}({\gamma}' = 1) = \frac{1}{{\alpha}'}\left({\mu}' - 1 \right)$ and ${}_{1}F_{1}({\gamma}'+1 =2) = \frac{1}{({\alpha}')^2}\left({\mu}'({\alpha}'-1)+1\right)$. Now consider:
\begin{align}
    {}_{2}F_{2}({\gamma}' = 1) & = \sum_{i = 0}^\infty \left(\frac{1}{1+i} \right)^2 \frac{{\alpha}'}{i!}\nonumber\\
    & = \frac{1}{\alpha'} \sum_{j = 1}^\infty \frac{1}{j} \frac{(\alpha')^i}{j!} \nonumber \\
    & = \frac{1}{\alpha'} \left(\text{Ei}(\alpha') - \gamma^* - \log\vert \alpha '\vert \right),  \nonumber
\end{align}
where $\text{Ei}(\alpha')$ is the exponential integral with $\gamma^*$ being the Euler-Mascheroni constant. Hence from Equation \ref{eq: rho lomax}, 
\begin{align}
    \rho\left(\alpha', \beta', \gamma' = 1, \delta', \eta' = 1 \right) &= \frac{\beta'\!\left(\mu'-1-\dfrac{(\alpha'-1)\mu'}{\alpha'}-\dfrac{1}{\alpha'}\right)}
{\sqrt{\alpha'\Biggl( \beta'\mu'\Bigl(\mu'-\dfrac{\mu'}{\alpha}+\dfrac{1}{\alpha}\Bigr)
+\delta' (\mu')^2
+\dfrac{(\beta')^{2}\Bigl(\mu'\bigl(-\log(\alpha')+\mathrm{Ei}(\alpha')-\gamma^*\bigr)-\dfrac{(1-\mu')^{2}}{\alpha'}\Bigr)}{\alpha'}\;\Biggr)}}. \nonumber
\end{align}
Now after taking the limit:
\begin{align}
    \lim_{\beta' \to \infty} \rho\left(\alpha', \beta', \gamma' = 1, \delta',\eta' = 1 \right) & =
\frac{\sqrt{\alpha'} \left(\mu' - 1 - \frac{(\alpha' - 1)\mu'}{\alpha'} - \frac{1}{\alpha'}\right)}
{\sqrt{A}},  \label{eq: case2 rho bounds 1 lomax}
\end{align}
and
\begin{align}
    \lim_{\beta' \to -\infty} \rho\left(\alpha', \beta', \gamma' = 1, \delta',\eta' = 1\right) & =
-\frac{\sqrt{\alpha'} \left(\mu' - 1 - \frac{(\alpha' - 1)\mu'}{\alpha'} - \frac{1}{\alpha'}\right)}
{\sqrt{A}}, \label{eq: case2 rho bounds 2 lomax}
\end{align}
with $A = - \alpha' \mu' \log\vert\alpha'\vert + \alpha' \mu' \operatorname{Ei}(\alpha') - \gamma^* \alpha' \mu' - (\mu')^2 + 2 \mu' - 1$.
Afterwhich, taking $\lim_{\alpha' \to 0^+}$ to Equations \ref{eq: case2 rho bounds 1 lomax} and \ref{eq: case2 rho bounds 2 lomax} implies $-1<\rho <1$.

\subsection{Case IV}
\paragraph{$\beta' =1$}
Again, taking the limit of Equation \ref{eq: rho lomax}, and noting for positive correlation (that is, $\beta' >0$) we must have $\delta' \geq 0$:
\begin{equation}
    \lim_{\delta' \to 0^+}\rho(\alpha',\beta' = 1, \gamma' = 1, \delta', \eta' = 1)  = \frac{\mu' - 1 + \frac{(1 - \alpha)\mu'}{\alpha} - \frac{1}{\alpha}}
{\sqrt{A+(\alpha'\mu')^2 - \alpha'(\mu')^2 + \alpha'\mu'}}.
    \label{eq: case4 rho bounds 1 lomax}
\end{equation}
Noting the monotonicity of Equation \ref{eq: case4 rho bounds 1 lomax}, we take $\lim_{\alpha' \to 0^+}$ to obtain $\frac{\sqrt{3}}{3} \approx 0.57735$. Hence, our bounds are $ 0<\rho < \frac{\sqrt{3}}{3} \approx 0.57735$ also illustrated in Figure \ref{fig: rho case 4 pos lomax} displaying $\rho(\alpha',\beta' = 1, \gamma' = 1, \delta', \eta' = 1)$.

\paragraph{$\beta =-1$}
Again, taking the limit of Equation \ref{eq: rho lomax}, and noting for negative correlation (that is, $\beta' <0$) we must have $\delta' \geq -\beta$:
\begin{equation}
    \lim_{\delta' \to 1^+}\rho(\alpha',\beta' = 1, \gamma' = 1, \delta', \eta' = 1)  = -\frac{ \sqrt{\alpha'}\left(\mu' - 1 - \frac{(1 - \alpha)\mu'}{\alpha} - \frac{1}{\alpha} \right)}
{\sqrt{A+ \alpha'(\mu')^2 - \alpha'\mu'}}  \label{eq: case4 neg rho bounds 1 lomax}
\end{equation}
Now Equation \ref{eq: case4 neg rho bounds 1 lomax} is not monotonic with respect to $\alpha$, so instead of taking limits, we rather minimise computationally: $(\alpha')^* = \operatorname*{argmin}_{\alpha'}  \left\{\lim_{\delta'\to 1}\rho(\alpha',\beta' = 1, \gamma' = 1, \delta', \eta' = 1) \right\} = 1.2345$, such that $\left\{\lim_{\delta'\to 1}\rho(\alpha',\beta' = 1, \gamma' = 1, \delta', \eta' = 1) \right\}_{ \rvert \alpha' = (\alpha')^*} \approx -0.31465$. Hence, our bounds are $ \approx -0.31465<\rho <0$ also illustrated in Figure \ref{fig: rho case 4 neg lomax} displaying $\rho(\alpha',\beta' = -1, \gamma' = 1, \delta', \eta' = 1)$.

\begin{figure}[H]
  \centering
  \begin{subfigure}{0.48\textwidth}
    \centering
    \includegraphics[width=\linewidth]{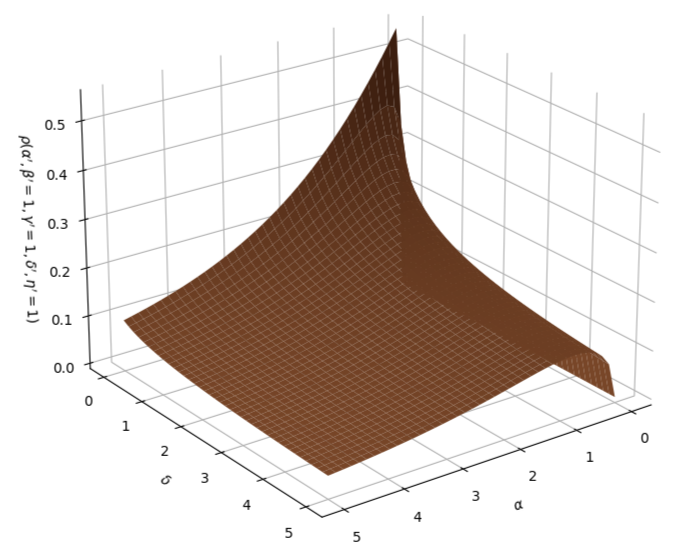}
    \caption{$\rho(\alpha',\beta' = 1, \gamma' = 1, \delta', \eta' = 1)$ against $\alpha'>0$ and $\delta' \geq 0$ noting a maximum achieved at $\frac{\sqrt{3}}{3}$.}
    \label{fig: rho case 4 pos lomax}
  \end{subfigure}
  \hfill
  \begin{subfigure}{0.48\textwidth}
    \centering
    \includegraphics[width=\linewidth]{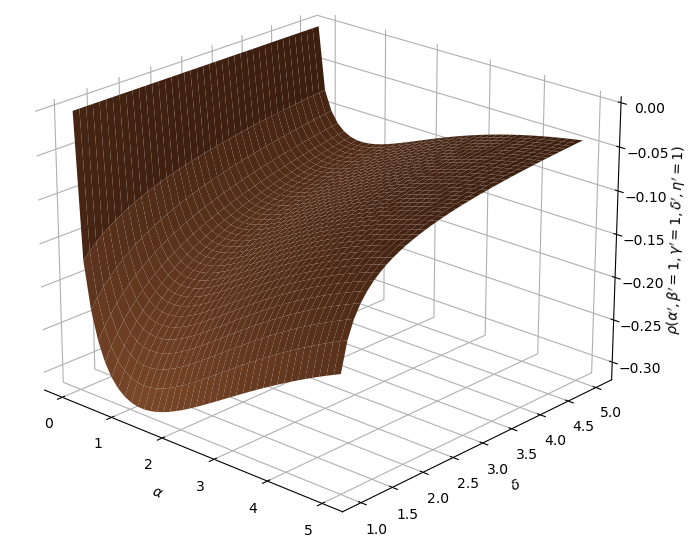}
    \caption{$\rho(\alpha',\beta' = -1, \gamma' = 1, \delta', \eta' = 1)$ against $\alpha' >0$, $\delta' \geq 1$ noting a minimum achieved at $\approx -0.314$.}
    \label{fig: rho case 4 neg lomax}
  \end{subfigure}
  \caption{$\rho(\alpha',\beta' = 1, \gamma' = 1, \delta',\eta' = 1)$ and $\rho(\alpha',\beta' = -1, \gamma' = 1, \delta', \eta' = 1)$.}
\end{figure}

\subsection{Case V}
Similarly, taking the limits of Equation \ref{eq: rho lomax}:
\begin{equation}
    \lim_{\beta'\to \infty}\lim_{\alpha'\to 0^+}\rho(\alpha',\beta', \gamma' = 1, \delta' = 0)  = 1, \nonumber
\end{equation}
and 
\begin{equation}
    \lim_{\beta'\to -\infty}\lim_{\alpha'\to 0^+}\rho(\alpha',\beta', \gamma' = 1, \delta' = 0)  = -1.\nonumber
\end{equation}
Hence, the bounds are $-1 < \rho < 1$.\\

We summarise the correlation $\rho(\alpha', \beta', \gamma', \delta', \eta')$ of Equation \ref{eq: rho lomax} bounds in Table \ref{table: rho bounds lomax}, confirmed with numerical optimisation.
\begin{table}[H]
    \centering
    \begin{tabular}{lcc}
        \hline
        Model & $\rho$ Lower Bound &  $\rho$ Upper Bound\\ \hline
        Full Lomax, Lomax($\eta' = 1$, $\gamma'$) (including Case III) & -1 & 1 \\ 
        Case I ($\beta' = 1$) & 0& $\frac{\sqrt{2}}{2} \approx 0.707$  \\
        Case I ($\beta' = -1$) & $\approx-0.482$ &  0 \\ 
        Case II ($\gamma' = 1$)& -1 & 1 \\
        Case IV ($\beta' = 1, \gamma' = 1$)  & 0 & $\frac{\sqrt{3}}{3} \approx 0.577$ \\ 
        Case IV ($\beta' = -1, \gamma' = 1$)  & $\approx -0.315$ & 0 \\
        Case V ($\delta' = 0, \gamma' = 1$)  & -1 & 1 \\
        \hline
    \end{tabular}
    \caption{$\rho$ lower and upper bounds for the full lomax model and sub-models.}
    \label{table: rho bounds lomax}
\end{table}

\subsection{Maximum likelihood estimation}
Consider again, the data in the form of $\textbf{X}^{(1)}, \textbf{X}^{(2)}, ..., \textbf{X}^{(n)}$ (where for each $i = 1, \ldots, n$, $\textbf{X}^{(i)} = (X_{1i}, X_{2i})^{T} )$ are obtained, which are independent and identically distributed according to Equations \ref{eq: lomax Function 1} - \ref{eq: lomax Function 2}, the the likelihood function is as follows from Equation \ref{eq:P(X1, X2) lomax}:
\begin{align}
    L(\alpha', \beta', \gamma', \delta', \eta') &= \prod_{i=1}^{n}  \frac{e^{-\alpha'}(\alpha')^{x_{1i}}}{x_{1i}!} \frac{e^{-\delta' - \beta' \left(1 -  \left(\frac{\gamma'}{x_{1i} +\gamma'}\right)^{\eta'}   \right)} \left\{\delta' + \beta' \left(1 - \left(\frac{\gamma'}{x_{1i} +\gamma'}\right)^{\eta'}\right) \right\}^{x_{2i}}       } {x_{2i}!} \nonumber \\
    &= \frac{e^{-n(\alpha' +\beta' + \delta')} (\alpha')^{\sum_{i= 1}^nx_{1i}} e^{\beta'\sum_{i = 1}^n \left(\frac{\gamma'}{x_{1i} + \gamma'} \right)^{\eta'}} \prod_{i = 1}^n  \left(\delta' + \beta' \left(1 - \left(\frac{\gamma'}{x_{1i} + \gamma'} \right)^{\eta'}\right)\right)^{x_{2i}} }{\prod_{i = 1}^n  \left( x_{1i}!x_{2i}!\right)}. \nonumber
\end{align}
Now the log-likelihood is:
\begin{align}
    l = \log \left(L(\alpha', \beta', \gamma', \delta', \eta') \right) = &-n(\alpha' +\beta' +\delta') + \log(\alpha') \sum_{i =1}^nx_{1i} + \beta' \sum_{i = 1}^n \left(\frac{\gamma'}{x_{1i} + \gamma'} \right)^{\eta'} \nonumber \\
    & + \sum_{i = 1}^n x_{2i}\log \left(\delta' + \beta' \left(1 - \left(\frac{\gamma'}{x_{1i} + \gamma'} \right)^{\eta'}\right) \right)
    +h\left(\mathbf{x}_1, \mathbf{x}_2 \right), \label{eq: log lik lomax}
\end{align}
where $h\left(\mathbf{x}_1, \mathbf{x}_2 \right) = -\log \left( \prod_{i = 1}^n  \left( x_{1i}!x_{2i}!\right)\right)$. Now since $\frac{\partial l}{\partial \alpha'} = 0 \implies \hat{\alpha}' = \frac{\sum_{i = 1}^nx_{1i}}{n} = \tilde{\alpha}'$, that is, the m.l.e of $\alpha'$ is equivalent to the m.m.e of $\alpha'$. The m.l.e's of $\beta', \gamma', \delta', \eta'$ are to be solved numerically using Equation \ref{eq: log lik lomax}.

\subsection{Likelihood ratio tests}
We note the natural parameter space of the full Lomax model is $\Theta = \{ (\alpha', \beta', \gamma', \delta',\eta' )^{T}: \alpha' > 0, \gamma' >  0, \beta' \in \mathbb{R} \setminus \{0\}, \delta' \geq \max(-\beta', 0) , \eta' > 0 \}$. Let $\hat{\alpha}', \hat{\beta}', \hat{\gamma}', \hat{\delta'}$ and $\hat{\eta}'$ be the m.l.e's of the full Lomax model. Furthermore, we denote $(\hat{\alpha}')^*, (\hat{\beta}')^*, (\hat{\gamma}')^*, (\hat{\delta}')^*$ and $(\hat{\eta}')^*$ as any m.l.e under $H_0$ in this section.

\subsubsection{Testing $\beta' = \pm 1$}
 Under $H_0$, the natural parameter space is $\Theta_0 = \{ (\alpha', \gamma', \delta', \eta' )^{T}: \alpha' > 0, \gamma' > 0, \delta' \geq \mathbb{I}\left(\beta' <0 \right), \eta' > 0 \}$. The generalized likelihood ratio test statistic in Equation \ref{eq: lik ratio test} is:

\begin{align}
    \Lambda = \frac{ \frac{e^{-n((\hat{\alpha}')^* +\text{sign}(\beta') + (\hat{\delta}')^*)} ((\hat{\alpha}')^*)^{\sum_{i= 1}^nx_{1i}} e^{\text{sign}(\beta')\sum_{i = 1}^n \left(\frac{(\hat{\gamma}')^*}{x_{1i} + (\hat{\gamma}')^*} \right)^{(\hat{\eta}')^*}} \prod_{i = 1}^n  \left((\hat{\delta}')^* + (\hat{\beta}')^* \left(1 - \left(\frac{(\hat{\gamma}')^*}{x_{1i} + (\hat{\gamma}')^*} \right)^{(\hat{\eta}')^*}\right)\right)^{x_{2i}} }{\prod_{i = 1}^n  \left( x_{1i}!x_{2i}!\right)}}{\frac{e^{-n(\hat{\alpha}' +\hat{\beta}' + \hat{\delta}')} (\hat{\alpha}')^{\sum_{i= 1}^nx_{1i}} e^{\hat{\beta}'\sum_{i = 1}^n \left(\frac{\hat{\gamma}'}{x_{1i} + \hat{\gamma}'} \right)^{\hat{\eta}'}} \prod_{i = 1}^n  \left(\hat{\delta}' + \hat{\beta}' \left(1 - \left(\frac{\hat{\gamma}'}{x_{1i} + \hat{\gamma}'} \right)^{\hat{\eta}'}\right)\right)^{x_{2i}} }{\prod_{i = 1}^n  \left( x_{1i}!x_{2i}!\right)}}. \nonumber
\end{align}
Since $(\hat{\alpha}')^* = \hat{\alpha}'$, we have:
\begin{eqnarray}
	\Lambda &=& e^{-n\left(\text{sign}(\beta') +(\hat{\delta}')^* -\hat{\beta}' - \hat{\delta}'  \right)}e^{\left({\text{sign}(\beta') \sum_{i = 1}^n  \left(\frac{(\hat{\gamma}')^*}{x_{1i} + (\hat{\gamma}')^*} \right)^{(\hat{\eta}')^*}
         } -\hat{\beta'}\sum_{i = 1}^n   \left(\frac{(\hat{\gamma}')}{x_{1i} +(\hat{\gamma}')} \right)^{\hat{\eta}'}
   \right)} \prod_{i = 1}^{n}  \left[ \frac{(\hat{\delta}')^* +  \text{sign}(\beta') \left(1 -   \left(\frac{(\hat{\gamma}')^*}{x_{1i} + (\hat{\gamma}')^*} \right)^{(\hat{\eta}')^*}
\right)} {\hat{\delta}' + \hat{\beta}' \left(1 -   \left(\frac{\hat{\gamma}'}{x_{1i} + \hat{\gamma}'} \right)^{\hat{\eta}'}
\right)} \right]^{x_{2i}}. \nonumber
\end{eqnarray}	
And by taking the logarithm, we obtain:
\begin{align}
	\log{\Lambda} = &-n\left(\text{sign}(\beta') + (\hat{\delta}')^* -\hat{\beta}' - \hat{\delta}'  \right) + {\text{sign}(\beta') \sum_{i = 1}^n  \left(\frac{(\hat{\gamma}')^*}{x_{1i} + (\hat{\gamma}')^*} \right)^{(\hat{\eta}')^*}
         } -\hat{\beta'}\sum_{i = 1}^n   \left(\frac{(\hat{\gamma}')}{x_{1i} +(\hat{\gamma}')} \right)^{\hat{\eta}'}
   \nonumber \\  &+ \sum_{i=1}^{n} x_{2i} \log{  \left[ \frac{(\hat{\delta}')^* +  \text{sign}(\beta') \left(1 -   \left(\frac{(\hat{\gamma}')^*}{x_{1i} + (\hat{\gamma}')^*} \right)^{(\hat{\eta}')^*}
\right)} {\hat{\delta}' + \hat{\beta}' \left(1 -   \left(\frac{\hat{\gamma}'}{x_{1i} + \hat{\gamma}'} \right)^{\hat{\eta}'}
\right)} \right]}. \nonumber
\end{align}	

\subsubsection{Testing $\gamma' =  1$}
Under $H_0$, the natural parameter space is $\Theta_0 = \{ (\alpha', \beta', \delta', \eta' )^{T}: \alpha' > 0, \beta' \in \mathbb{R} \setminus \{0\}, \delta' \geq \max(-\beta, 0), \eta'>0 \ \}$. The logarithm of the generalized likelihood ratio test statistic in Equation \ref{eq: lik ratio test} is:
\begin{align}
	\log{\Lambda} = &-n\left( (\hat{\beta}')^* + (\hat{\delta}')^* -\hat{\beta}' - \hat{\delta}'  \right) + (\hat{\beta}')^* \sum_{i = 1}^n  \left(\frac{1}{x_{1i} +1} \right)^{(\hat{\eta}')^*}
         -\hat{\beta'}\sum_{i = 1}^n   \left(\frac{(\hat{\gamma}')}{x_{1i} +(\hat{\gamma}')} \right)^{\hat{\eta}'}
   \nonumber \\  &+ \sum_{i=1}^{n} x_{2i} \log{  \left[ \frac{(\hat{\delta}')^* +  (\hat{\beta}')^* \left(1 -   \left(\frac{1}{x_{1i} + 1} \right)^{(\hat{\eta}')^*}
\right)} {\hat{\delta}' + \hat{\beta}' \left(1 -   \left(\frac{\hat{\gamma}'}{x_{1i} + \hat{\gamma}'} \right)^{\hat{\eta}'}
\right)} \right]}. \nonumber
\end{align}	

\subsubsection{Testing $\delta' =  0$}
Under $H_0$, the natural parameter space is $\Theta_0 = \{ (\alpha', \beta', \gamma', \eta' )^{T}: \alpha' > 0, \beta' \in \mathbb{R} \setminus \{0\}, \gamma' >0, \eta' > 0  \}$. The logarithm of the generalized likelihood ratio test statistic in Equation \ref{eq: lik ratio test} is:
\begin{align}
	\log{\Lambda} = &-n\left( (\hat{\beta}')^* -\hat{\beta}' - \hat{\delta}'  \right) + {(\hat{\beta}')^* \sum_{i = 1}^n  \left(\frac{(\hat{\gamma}')^*}{x_{1i} + (\hat{\gamma}')^*} \right)^{(\hat{\eta}')^*}
         } -\hat{\beta'}\sum_{i = 1}^n   \left(\frac{(\hat{\gamma}')}{x_{1i} +(\hat{\gamma}')} \right)^{\hat{\eta}'}
   \nonumber \\  &+ \sum_{i=1}^{n} x_{2i} \log{  \left[ \frac{ (\hat{\beta}')^*\left(1 -   \left(\frac{(\hat{\gamma}')^*}{x_{1i} + (\hat{\gamma}')^*} \right)^{(\hat{\eta}')^*}
\right)} {\hat{\delta}' + \hat{\beta}' \left(1 -   \left(\frac{\hat{\gamma}'}{x_{1i} + \hat{\gamma}'} \right)^{\hat{\eta}'}
\right)} \right]}. \nonumber
\end{align}	

\subsubsection{Testing $\beta' = \pm 1$ and $\gamma' = 1$}
Under $H_0$, the natural parameter space is $\Theta_0 = \{ (\alpha', \delta', \eta' )^{T}: \alpha' > 0, \delta' \geq \mathbb{I}\left(\beta' <0 \right), \eta' > 0\}$. The logarithm of the generalized likelihood ratio test statistic in Equation \ref{eq: lik ratio test} is:
\begin{align}
	\log{\Lambda} = &-n\left( \text{sign}(\beta')+ (\hat{\delta}')^*-\hat{\beta}' - \hat{\delta}'  \right) + \text{sign}(\beta')\sum_{i = 1}^n  \left(\frac{1}{x_{1i} +1} \right)^{(\hat{\eta}')^*}
         -\hat{\beta'}\sum_{i = 1}^n   \left(\frac{(\hat{\gamma}')}{x_{1i} +(\hat{\gamma}')} \right)^{\hat{\eta}'}
   \nonumber \\  &+ \sum_{i=1}^{n} x_{2i} \log{  \left[ \frac{(\hat{\delta}')^* +  \text{sign}(\beta') \left(1 -   \left(\frac{1}{x_{1i} + 1} \right)^{(\hat{\eta}')^*}
\right)} {\hat{\delta}' + \hat{\beta}' \left(1 -   \left(\frac{\hat{\gamma}'}{x_{1i} + \hat{\gamma}'} \right)^{\hat{\eta}'}
\right)} \right]}. \nonumber
\end{align}	

\subsubsection{Testing $\delta' =  0$ and $\gamma' = 1$}
Under $H_0$, the natural parameter space is $\Theta_0 = \{ (\alpha', \beta', \eta' )^{T}: \alpha' > 0, \beta' \in \mathbb{R} \setminus \{0\}, \eta' > 0  \}$. The logarithm of the generalized likelihood ratio test statistic in Equation \ref{eq: lik ratio test} is:
\begin{align}
	\log{\Lambda} = &-n\left( (\hat{\beta}')^* -\hat{\beta}' - \hat{\delta}'  \right) + {(\hat{\beta}')^* \sum_{i = 1}^n  \left(\frac{1}{x_{1i} + 1} \right)^{(\hat{\eta}')^*}
         } -\hat{\beta'}\sum_{i = 1}^n   \left(\frac{(\hat{\gamma}')}{x_{1i} +(\hat{\gamma}')} \right)^{\hat{\eta}'}
   \nonumber \\  &+ \sum_{i=1}^{n} x_{2i} \log{  \left[ \frac{ (\hat{\beta}')^*\left(1 -   \left(\frac{1}{x_{1i} + 1} \right)^{(\hat{\eta}')^*}
\right)} {\hat{\delta}' + \hat{\beta}' \left(1 -   \left(\frac{\hat{\gamma}'}{x_{1i} + \hat{\gamma}'} \right)^{\hat{\eta}'}
\right)} \right]}. \nonumber
\end{align}	

\subsubsection{Testing $\eta' =  1$}
Under $H_0$, the natural parameter space is $\Theta_0 = \{ (\alpha', \beta', \gamma', \delta' )^{T}: \alpha' > 0, \beta' \in \mathbb{R} \setminus \{0\}, \gamma' >0,  \delta' \geq \max(-\beta, 0) \}$. The logarithm of the generalized likelihood ratio test statistic in Equation \ref{eq: lik ratio test} is:
\begin{align}
	\log{\Lambda} = &-n\left( (\hat{\beta}')^*+ (\hat{\delta}')^* -\hat{\beta}' - \hat{\delta}'  \right) + {(\hat{\beta}')^* \sum_{i = 1}^n  \frac{(\hat{\gamma}')^*}{x_{1i} + (\hat{\gamma}')^*} 
         } -\hat{\beta'}\sum_{i = 1}^n   \left(\frac{(\hat{\gamma}')}{x_{1i} +(\hat{\gamma}')} \right)^{\hat{\eta}'}
   \nonumber \\  &+ \sum_{i=1}^{n} x_{2i} \log{  \left[ \frac{ (\hat{\delta}')^*+ (\hat{\beta}')^*\left(1 -  \frac{(\hat{\gamma}')^*}{x_{1i} + (\hat{\gamma}')^*} 
\right)} {\hat{\delta}' + \hat{\beta}' \left(1 -   \left(\frac{\hat{\gamma}'}{x_{1i} + \hat{\gamma}'} \right)^{\hat{\eta}'}
\right)} \right]}. \nonumber
\end{align}

\section{Simulation Data}
Simulating from the pseudo-models is straightforward because of their marginal-conditional structure. For example, for the full exponential model given $(\alpha,\beta,\gamma,\delta)$, we generate a bivariate sample $(x_1,x_2)$ as follows:
\begin{enumerate}[label=\emph{Step \Roman*:}]
  \item Simulate $x_1$ from $\text{Poisson}(\alpha)$.
  \item Simulate $x_2$ from $\text{Poisson}\bigl(\delta + \beta (1 - e^{\gamma x_1})\bigr)$.
\end{enumerate}

We also simulate data for the exponential Case I ($\rho <0$) sub-model by fixing $\beta=-1$. An analogous simulation procedure applies to the Lomax model and its sub-models. Furthermore, we simulate $10000$ data sets of size $n = 100, 1000$ and $10000$ for fixed $\alpha, \beta, \gamma$ and $\delta$ values. \\

For the full exponential model (Table~\ref{Table: Simulations for full exp model}), recall that the bounds implied by $M_1$ on $\tfrac{S_{12}^2}{S_2-M_2}$ in Equation~\ref{eq: full bounds 2} may lead to the non-existence of the m.m.e. $\tilde{\gamma}$ for certain combinations of $M_1$, $S_2$, $M_2$, and $S_{12}$ (see Figure~\ref{fig: nu mme full exp}). To avoid this issue, we select parameter values that satisfy these bounds, specifically: $\alpha = 5, \beta = -20, \gamma = 0.5$ and $\delta = 25$. An analogous construction applies to the Lomax$(\eta' = 1, \gamma')$ model (Table~\ref{Table: Simulations for Lomax model}), where existence of $\tilde{\gamma}'$ is governed by the bounds in Equation~\ref{eq: full bounds 2 lomax}. Similarly, recall for the exponential Case I ($\rho < 0$) sub-model (Table~\ref{Table: Simulations for exp case 1}), the existence of a real-valued solution for $\tilde{\nu}$ requires that $\frac{M_1}{S_{12}}\text{sign}(\beta) \geq e$ -  we therefore select parameter values to ensure this condition is met, setting them as: $\alpha = 1, \gamma = 0.3$ and $ \delta = 25$.\\

We omit simulation results for the remaining sub-models, as no bounds are required to guarantee the existence of their m.m.e's; consequently, their corresponding simulations are straightforward. Furthermore, recall that no formal existence bounds for $\tilde{\gamma}'$ could be derived for the Lomax Case I sub-model; instead, Figure~\ref{fig: mme gamma lomax case 1} was provided solely to facilitate a visual illustration of the existence of $\tilde{\gamma}'$. Accordingly, simulations for the Lomax Case I sub-model are not reported here - in practical applications, the existence of $\tilde{\gamma}'$ would depend exclusively on the convergence of the optimisation routine employed.\\

The corresponding m.m.e and m.l.e results are reported in Tables~\ref{Table: Simulations for full exp model} and \ref{Table: Simulations for exp case 1} for the full exponential model and exponential Case I sub-model, and Table \ref{Table: Simulations for Lomax model} for the  Lomax$(\eta' = 1, \gamma')$ model. Included are $95\%$ confidence intervals computed as $\hat{\theta} \pm Z_{\alpha/2}\text{S.E.}(\hat{\theta})$\footnote{$Z_{\alpha/2}$ denotes the $100(1 - \alpha/2)$\textsuperscript{th} percentile of the standard normal distribution.}, where $\hat{\theta}$ is the point estimator (as done in \citet{arnold2021statistical}). As anticipated, the standard errors of both the m.m.e. and m.l.e. decrease with increasing sample size $n$, resulting in correspondingly narrower $95\%$ confidence intervals. For the full exponential model (Table~\ref{Table: Simulations for full exp model}) and the Lomax$(\eta' = 1, \gamma')$ model (Table~\ref{Table: Simulations for Lomax model}), the m.l.e.'s consistently outperform the m.m.e.'s, yielding smaller standard errors and tighter confidence intervals. In contrast, for the exponential Case I sub-model ($\rho<0$), the reverse pattern is observed, with the m.m.e.'s demonstrating greater accuracy. A salient feature in Table~\ref{Table: Simulations for exp case 1} is the relative instability of the estimators of $\gamma$: reliable estimation of $\gamma$ appears to require substantially larger sample sizes. Finally, we note that the Pearson correlation (PC) converges to the population correlation as $n$ increases.

\begin{table}[H] \centering 
	\begin{tabular}{@{\extracolsep{1pt}} ccccccccc} 
		\\[-1.8ex]\hline 
		\hline \\[-1.8ex] 
		n & Parameter & Moment & MLE & SE(Moment) & SE(MLE) & $95\%$CI (MM) & $95\%$CI (MM) & PC \\ 
		\hline \\[-1.8ex] 
		& $\alpha$ & 5.000 & 5.000 & 0.217 & 0.217 & $(4.574,\,5.426)$ & $(4.574,\,5.426)$ & \\
& $\beta$  & -19.493 & -20.269 & 4.533 & 3.512 & $(-28.378,\,-10.607)$ & $(-27.153,\,-13.385)$ & \\
100 & $\gamma$ & 0.474 & 0.504 & 0.164 & 0.102 & $(0.154,\,0.795)$ & $(0.304,\,0.703)$ & -0.601 \\
& $\delta$ & 23.929 & 25.160 & 5.078 & 3.827 & $(13.977,\,33.882)$ & $(17.660,\,32.661)$ & \\
& $\rho$   & -0.596 & -0.597 & 0.075 & 0.044 & $(-0.743,\,-0.449)$ & $(-0.683,\,-0.511)$ & 
\\[1ex]
\hline 
& $\alpha$ &5.000& 5.000 & 0.069 & 0.069 & $(4.864,\,5.134)$ & $(4.864,\,5.134)$ & \\
& $\beta$  & -19.872 & -20.010 & 1.638 & 1.082 & $(-23.083,\,-16.660)$ & $(-22.131,\,-17.889)$ & \\
1000 & $\gamma$ & 0.496 & 0.500 & 0.054 & 0.032 & $(0.391,\,0.601)$ & $(0.437,\,0.563)$ & -0.599 \\
& $\delta$ & 24.853 & 25.002 & 1.768 & 1.165 & $(21.388,\,28.319)$ & $(22.718,\,27.285)$ & \\
& $\rho$   & -0.598 & -0.598 & 0.024 & 0.014 & $(-0.646,\,-0.550)$ & $(-0.625,\,-0.571)$ &
\\[1ex]
\hline  
& $\alpha$ & 5.000 & 5.000 & 0.022 & 0.022 & $(4.957,\,5.042)$ & $(4.957,\,5.042)$ & \\
& $\beta$  & -19.996 & -20.002 & 0.523 & 0.348 & $(-21.022,\,-18.970)$ & $(-20.683,\,-19.320)$ & \\
10000& $\gamma$ & 0.500 & 0.500 & 0.017 & 0.010 & $(0.466,\,0.533)$ & $(0.480,\,0.520)$ & -0.598 \\
& $\delta$ & 24.990 & 25.003 & 0.563 & 0.375 & $(23.885,\,26.094)$ & $(24.268,\,25.739)$ & \\
& $\rho$   & -0.598 & -0.598 & 0.008 & 0.004 & $(-0.613,\,-0.583)$ & $(-0.606,\,-0.589)$ & \\
		\hline
	\end{tabular} 
    \caption{Simulations for full exponential model ($\alpha = 5$, $\beta = -20$,  $\gamma = 0.5$, $\delta = 25$).} 
    \label{Table: Simulations for full exp model} 
\end{table}

\begin{table}[H] \centering 
	\begin{tabular}{@{\extracolsep{1pt}} ccccccccc} 
		\\[-1.8ex]\hline 
		\hline \\[-1.8ex] 
		n & Parameter & Moment & MLE & SE(Moment) & SE(MLE) & $95\%$CI (MM) & $95\%$CI (MM) & PC \\ 
		\hline \\[-1.8ex] 
& $\alpha$ & 1.000 & 1.000 & 0.096 & 0.096 & $(0.812,\,1.188)$ & $(0.812,\,1.188)$ & \\
100 & $\gamma$ & 0.247 & 0.196 & 0.273 & 1.506 & $(-0.288,\,0.782)$ & $(-2.755,\,3.147)$ &  -0.038 \\
& $\delta$ & 24.997 & 25.054 & 0.504 & 0.544 & $(24.010,\,25.985)$ & $(23.989,\,26.119)$ & \\
& $\rho$   & -0.037 & -0.040 & 0.021 & 0.027 & $(-0.078,\,0.004)$ & $(-0.093,\,0.012)$ & 
\\[1ex]
\hline 
& $\alpha$ & 1.000 & 1.000 & 0.031 & 0.031 & $(0.940,\,1.060)$ & $(0.940,\,1.060)$ & \\
1000 & $\gamma$ & 0.265 & 0.214 & 0.240 & 0.282 & $(-0.206,\,0.736)$ & $(-0.338,\,0.767)$ & -0.038\\
& $\delta$ & 24.997 & 24.988 & 0.206 & 0.257 & $(24.593,\,25.402)$ & $(24.485,\,25.492)$ & \\
& $\rho$   & -0.039 & -0.037 & 0.019 & 0.023 & $(-0.077,\,-0.001)$ & $(-0.083,\,0.009)$ & 
\\[1ex]
\hline  
& $\alpha$ & 1.000 & 1.000 & 0.010 & 0.010 & $(0.981,\,1.019)$ & $(0.981,\,1.019)$ & \\
10000& $\gamma$ & 0.297 & 0.251 & 0.113 & 0.195 & $(0.076,\,0.518)$ & $(-0.131,\,0.633)$ & -0.040 \\
& $\delta$ & 25.002 & 24.975 & 0.082 & 0.154 & $(24.842,\,25.163)$ & $(24.672,\,25.277)$ & \\
& $\rho$   & -0.040 & -0.037 & 0.010 & 0.019 & $(-0.059,\,-0.021)$ & $(-0.074,\,-0.001)$ & \\
\hline
	\end{tabular} 
    \caption{Simulations for exponential Case I ($\rho <0$) ($\alpha = 1$, $\gamma = 0.3$, $\delta = 25$).} 
    \label{Table: Simulations for exp case 1} 
\end{table}

\begin{table}[H] \centering 
	\begin{tabular}{@{\extracolsep{1pt}} ccccccccc} 
		\\[-1.8ex]\hline 
		\hline \\[-1.8ex] 
		n & Parameter & Moment & MLE & SE(Moment) & SE(MLE) & $95\%$CI (MM) & $95\%$CI (MM) & PC \\ 
		\hline \\[-1.8ex] 
        & $\alpha'$ & 4.990 & 4.990 & 0.217 & 0.217 & $(4.565,\,5.415)$ & $(4.565,\,5.415)$ & \\
& $\beta'$  & -19.231 & -19.379 & 8.210 & 5.874 & $(-35.323,\,-3.139)$ & $(-30.892,\,-7.865)$ & \\
100 & $\gamma'$ & 0.491 & 0.428 & 0.355 & 0.292 & $(-0.205,\,1.186)$ & $(-0.144,\,1.000)$ &  -0.379\\
& $\delta'$ & 24.133 & 24.263 & 8.235 & 6.456 & $(7.991,\,40.274)$ & $(11.610,\,36.917)$ & \\
& $\rho$   & -0.377 & -0.367 & 0.095 & 0.065 & $(-0.564,\,-0.191)$ & $(-0.495,\,-0.240)$ & 
\\[1ex]
\hline 
& $\alpha'$ & 4.999 & 4.999 & 0.069 & 0.069 & $(4.864,\,5.134)$ & $(4.865,\,5.133)$ & \\
& $\beta'$  & -19.633 & -19.987 & 3.546 & 1.882 & $(-26.582,\,-12.683)$ & $(-23.676,\,-16.299)$ & \\
1000 & $\gamma'$ & 0.507 & 0.497 & 0.135 & 0.083 & $(0.242,\,0.772)$ & $(0.335,\,0.660)$ & -0.372\\
& $\delta'$ & 24.623 & 24.995 & 3.552 & 1.916 & $(17.661,\,31.585)$ & $(21.239,\,28.752)$ & \\
& $\rho$   & -0.371 & -0.371 & 0.030 & 0.021 & $(-0.431,\,-0.312)$ & $(-0.412,\,-0.330)$ &
\\[1ex]
\hline 
& $\alpha'$ & 5.000 & 5.000 & 0.022 & 0.022 & $(4.957,\,5.042)$ & $(4.957,\,5.042)$ & \\
& $\beta'$  & -19.987 & -20.003 & 1.107 & 0.576 & $(-22.158,\,-17.816)$ & $(-21.131,\,-18.875)$ & \\
10000 & $\gamma'$ & 0.500 & 0.500 & 0.040 & 0.026 & $(0.421,\,0.580)$ & $(0.449,\,0.551)$ & -0.372\\
& $\delta'$ & 24.998 & 25.001 & 1.108 & 0.585 & $(22.826,\,27.171)$ & $(23.854,\,26.148)$ & \\
& $\rho$   & -0.372 & -0.371 & 0.010 & 0.007 & $(-0.391,\,-0.353)$ & $(-0.385,\,-0.358)$ & \\
		\hline
	\end{tabular} 
    \caption{Simulations for Lomax($\eta' = 1, \gamma'$) model ($\alpha' = 5$, $\beta' = -20$,  $\gamma' = 0.5$, $\delta' = 25$).} 
    \label{Table: Simulations for Lomax model} 
\end{table} 

\section{Application: Data set I}
Considering the data set that is in \citet{leiter1973some}, where the data is a $50$-mile stretch of Interstate $95$ in
Prince William, Stafford and Spottsylvania counties in Eastern Virginia. The data is categorized as fatal accidents, injury accidents or property damage accidents, together with corresponding number of fatalities and injuries for period 1 January 1969 to 31 October 1970. Number of fatalities was considered as $X_1$ since the Fisher Dispersion Index is $1.052$ and Number of injury accidents as $X_2$ (Fisher Dispersion Index is $1.14$) with $n = 639$. The Pearson Correlation Coefficient for the data is $0.205$ - indicating that $\beta$ or $\beta'$ should be set to $1$ in the Case I sub-models.\\

We observe that the bounds specified in Equations \ref{eq: full bounds 2} and \ref{eq: full bounds 2 lomax} are not satisfied for this dataset. Consequently, the m.m.e's, particularly $\tilde{\gamma}$ and $\tilde{\gamma}'$, for both the full exponential model and the Lomax$(\eta' = 1, \gamma')$ model, respectively, fail to exist, with the exception of $\tilde{\alpha}$ and $\tilde{\alpha}'$. Moreover, since $\frac{M_1}{S_{12}} < e$, the m.m.e's for the exponential Case I sub-model, specifically $\tilde{\gamma}$, do not exist. For the Lomax Case I sub-model, no explicit bounds governing the existence of m.m.e's (notably $\tilde{\gamma}'$) were derived; only Figure \ref{fig: mme gamma lomax case 1} provides a graphical reference. Therefore, the existence of $\tilde{\gamma}'$ in this case is determined solely by the convergence of the optimisation routine. Finally, as established earlier, the full Lomax model lacks sufficient number of moment equations to derive m.m.e's, explaining the absence of said estimators for this model.\\

Additionally, recall that for the Case III sub-models (no intercept), the condition \( x_{1}=0 \) implies \( x_{2}=0 \) with probability one. Consequently, any empirical observation satisfying \( x_{1}=0 \) but \( x_{2}>0 \) is assigned zero probability under the model, leading to a likelihood of zero and thus signaling model misspecification for such observations. This result implies that both Case~III sub-models, as well as the Case~V sub-models (no intercept), are unsuitable for application in this setting.\\

Furthermore, the \(\rho\) estimate based on the m.l.e.'s of the full Lomax model is excluded, as no computational methods exist for evaluating the correlation in Equation~\ref{eq: rho lomax} when the discreteness assumption on \(\eta'\) is relaxed. In such cases, one must approximate infinite sums without the simplifying closed-form representations provided by hypergeometric functions.\\

The m.m.e.'s, m.l.e.'s, AIC values, and the corresponding \(-2 \log \Lambda\) statistics (where applicable) for the various models are reported in Table~\ref{table: traffic exp} and Table~\ref{table: traffic lomax} for the exponential and Lomax models, respectively. For the Case~I and Case~II exponential sub-models, the \(-2 \log \Lambda\) values did not exceed \(\chi^2_{1,0.05} = 3.84\). Similarly, the Case~IV exponential sub-model's \(-2 \log \Lambda\) value did not exceed \(\chi^2_{2,0.05} = 5.99\). These results indicate insufficient evidence to reject the null hypothesis of no difference between the respective exponential sub-models and the full exponential model. \\ 

An analogous conclusion holds for the Lomax sub-models: the \(-2 \log \Lambda\) value for the Lomax\((\eta' = 1, \gamma')\) model is compared against \(\chi^2_{1,0.05} = 3.84\); for the Case~I and Case~II Lomax sub-models, the comparisons are against \(\chi^2_{2,0.05} = 5.99\); and for the Case~IV Lomax sub-model, the comparison is against \(\chi^2_{3,0.05} = 7.815\). In all cases, the corresponding \(-2 \log \Lambda\) values fail to exceed the critical thresholds, implying no significant difference between the full model and the sub-models.\\

The best-fitting models, according to the AIC criterion, are the exponential Case~IV sub-model and the Lomax Case~II sub-model. However, neither of these models outperforms \citet{arnold2021statistical}'s Bivariate Pseudo-Poisson Sub-Model~I (A\&M BPP SM-I) - a two-parameter model - with an AIC value of \(1866.094\). Figure~\ref{fig: traffic non-mirrored} depicts the regression functions \(E(X_2 \mid X_1 = x_1)\) against $X_1 = x_1$ for the aforementioned sub-models, as well as for A\&M BPP SM-I.

\begin{table}[H]
    \centering
    \begin{tabular}{cccccc}
    \hline
     Model & Parameter & m.m.e & m.l.e  & $-2 \log \Lambda$& AIC \\
     \hline
    Full exponential  &  $\alpha$ & 0.058 & 0.058 & - & 1869.831\\
              & $\beta$ & - & 1.240 & &\\
              & $\gamma$ & - & 1.256 & &\\
              & $\delta$ & - & 0.813 & &\\
              & $\rho$ &- & 0.215 & & \\
              \hline
    Case I ($\beta = 1$) & $\alpha$ & 0.058 & 0.058 & 0.04 & 1867.872\\
              & $\gamma$ & - & 2.175 &  &\\
              & $\delta$ & - & 0.813 & &\\
              & $\rho$ &- & 0.213 & & \\
              \hline
    Case II ($\gamma = 1$) & $\alpha$ & 0.058 & 0.058 & $\approx0$ & 1867.829\\
              & $\beta$ & 1.423 & 1.296 & &\\
              & $\delta$ & 0.811 & 0.813 & &\\
              & $\rho$ & 0.219 & 0.215 & & \\
              \hline
    Case IV ($\beta = 1 \ \& \ \gamma = 1$) & $\alpha$ & 0.058 & 0.058 & 1.437 & $\textbf{1867.268}$\\
              & $\delta$ & 0.826 & 0.821 & &\\
              & $\rho$ & 0.156 & 0.156 & & \\
              \hline
    \end{tabular}
    \caption{m.m.e's, m.l.e's, $-2 \log \Lambda$'s (where applicable) and AIC's for exponential model and sub-models.}
    \label{table: traffic exp}
\end{table}

\begin{table}[H]
    \centering
    \begin{tabular}{cccccc}
    \hline
     Model & Parameter & m.m.e & m.l.e  & $-2 \log \Lambda$& AIC \\
     \hline
    Full Lomax  &  $\alpha'$ & 0.058 & 0.058 & - & 1871.829\\
              & $\beta'$ & - & 1.775 & &\\
              & $\gamma'$ & - & 1.100 & &\\
              & $\delta'$ & - & 0.813 & &\\
              & $\eta'$ & - & 1.066 & &\\
              & $\rho$ &- & - & & \\
              \hline
 Lomax($\eta' = 1, \gamma'$)  &  $\alpha'$ & 0.058 & 0.058 & 0.004 & 1869.833\\
              & $\beta'$ & - & 1.533 & &\\
              & $\gamma'$ & - & 0.726 & &\\
              & $\delta'$ & - & 0.813 & &\\
              & $\rho$ &- & 0.215 & & \\
              \hline              
Case I ($\beta' = 1$) & $\alpha'$ & 0.058 & 0.058 & 0.066 & 1867.895 \\
              & $\gamma'$ & 0.091 & 0.121 &  &\\
              & $\delta'$ & 0.811 & 0.813 & &\\
              & $\rho$ & 0.219 & 0.214 & & \\
              \hline
Case II ($\gamma' = 1$) & $\alpha'$ & 0.058 & 0.058 & $\approx0$ & $\textbf{1867.829}$\\
              & $\beta'$ & 1.803 & 1.769 & &\\
              & $\delta'$ & 0.811 & 0.813 & &\\
              & $\rho$ & 0.219 & 0.215 & & \\
              \hline
    Case IV ($\beta' = 1 \ \& \ \gamma' = 1$) & $\alpha'$ & 0.058 & 0.058 & 3.581 & 1869.409\\
              & $\delta'$ & 0.834 & 0.826 & &\\
              & $\rho$ & 0.124 & 0.124 & & \\
              \hline
    \end{tabular}
    \caption{m.m.e's, m.l.e's, $-2 \log \Lambda$'s (where applicable) and AIC's for the Lomax model and sub-models.}
    \label{table: traffic lomax}
\end{table}

\begin{figure}[H]
    \centering
    \includegraphics[width=0.5\linewidth]{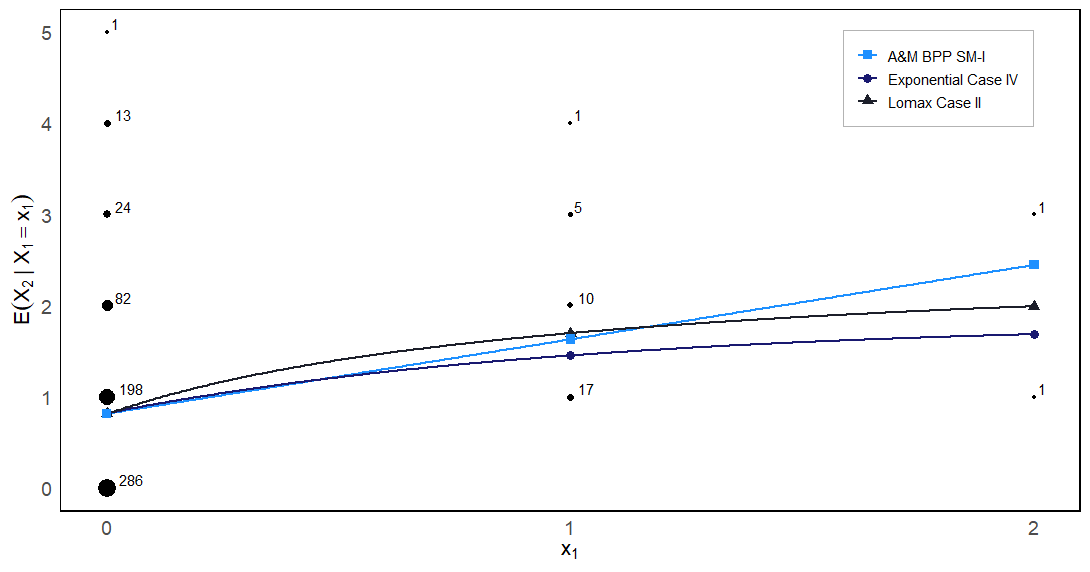}
    \caption{$E(X_2\mid X_1= x_1)$ vs. $X_1 = x_1$ for best fitted models using their m.l.e's. Size of black dots represent relative frequency of observation.}
    \label{fig: traffic non-mirrored}
\end{figure}

\subsection{Mirrored}
After assuming $X_1 \sim \text{Poisson}(\lambda_1)$ and $X_2\mid X_1 \sim \text{Poisson}(\lambda_2(x_1))$, it is natural to think $X_1$ influences $X_2$ in some way. However, we entertain the possibility that the data may be better modelled by the corresponding "mirrored"
model where $X_2 \sim \text{Poisson}(\lambda_1)$ and $X_1\mid X_2 \sim \text{Poisson}(\lambda_2(x_2))$, just as was undergone in \citet{arnold2021statistical}. Utilizing the mirrored models results in Number of Injury Accidents being $X_1$ (Fisher Dispersion Index of
$1.41$) and Number of Fatalities being $X_2$ (Fisher Dispersion Index of $1.052$).\\

We note again that, since the bounds in Equations~\ref{eq: full bounds 2} and \ref{eq: full bounds 2 lomax} are not satisfied for this mirrored dataset, the m.m.e.'s (particularly \(\tilde{\gamma}\) and \(\tilde{\gamma}'\)) for the mirrored full exponential model and the mirrored Lomax\((\eta' = 1, \gamma')\) model do not exist, with the exception of \(\tilde{\alpha}\) and \(\tilde{\alpha}'\). Because the m.m.e.'s for the mirrored Case~III sub-models are derived in the same manner, they are likewise omitted. Furthermore, we exclude the m.m.e.'s for the mirrored Case~IV sub-models, as these yield \(\tilde{\delta} < 0\) and \(\tilde{\delta}' < 0\), implying said sub-models are not appropriate for this dataset.\\

For the mirrored dataset, we observe that whenever \(X_1 = 0\), it follows that \(X_2 = 0\), and never \(X_2 > 0\). Consequently, it is expected that the models identify a no-intercept parameterization as appropriate, yielding estimates \(\tilde{\delta} = 0\) or \(\tilde{\delta}' = 0\), as consistently observed in Table~\ref{table: traffic mirror exp} and Table~\ref{table: traffic mirror lomax}.\\

The m.m.e.'s, m.l.e.'s, AIC values, and the corresponding \(-2 \log \Lambda\) statistics (where applicable) for the various mirrored models are presented in Table~\ref{table: traffic mirror exp} and Table~\ref{table: traffic mirror lomax}. For the mirrored Case~I, Case~II, and Case~III exponential sub-models, the \(-2 \log \Lambda\) values did not exceed \(\chi^2_{1,0.05} = 3.84\). Similarly, the \(-2 \log \Lambda\) value for the mirrored Case~V exponential sub-model did not exceed \(\chi^2_{2,0.05} = 5.99\). These results again provide insufficient evidence to reject the null hypothesis of no difference between said mirrored exponential sub-models and the mirrored full exponential model.\\  

An analogous conclusion applies to the mirrored Lomax sub-models: the \(-2 \log \Lambda\) statistic for the mirrored Lomax\((\eta' = 1, \gamma')\) model is compared to \(\chi^2_{1,0.05} = 3.84\); for the mirrored Case~I and Case~II Lomax sub-models, the comparisons are made against \(\chi^2_{2,0.05} = 5.99\); and for the mirrored Case~V Lomax sub-model, the comparison is made against \(\chi^2_{3,0.05} = 7.815\). In all instances, the \(-2 \log \Lambda\) values fail to exceed the respective critical values, indicating no significant difference between the mirrored full model and the mirrored sub-models.\\

We note, however, that both mirrored Case~IV sub-models differ significantly from their respective mirrored full models, as indicated by their \(-2 \log \Lambda\) values. This conclusion is further supported by the substantially higher AIC values observed for these mirrored sub-models relative to their corresponding mirrored full models.\\

Finally, we note that the mirrored exponential and mirrored Lomax Case~II, Case~III, and Case~V sub-models achieve lower AIC values than \citet{arnold2021statistical}'s Bivariate Pseudo-Poisson Mirrored Sub-Model~II (AIC = \(1847.505\)) - the best-fitting model reported by \citet{arnold2021statistical} for \citet{leiter1973some}'s accident and fatality dataset. Figure~\ref{fig: traffic mirrored} presents the regression forms \(E(X_2 \mid X_1 = x_1)\) for both the mirrored exponential and mirrored Lomax Case~V sub-models, as well as for A\&M BPP MSM-II. The mirrored exponential and mirrored Lomax sub-models exhibit superior fit, which we attribute to the greater curvature of their regression forms - illustrated in Figure~\ref{fig: traffic mirrored} - relative to \citet{arnold2021statistical}'s linear model.\\

Furthermore, the bivariate probability mass functions for the best-fitting models (both non-mirrored and mirrored) - computed using Equation~\ref{eq:P(X1, X2) exp} for the exponential model and Equation~\ref{eq:P(X1, X2) lomax} for the Lomax model - are presented in Figure~\ref{fig: hists traffic}.
 
\begin{table}[H]
    \centering
    \begin{tabular}{cccccc}
    \hline
     Model & Parameter & m.m.e & m.l.e  & $-2 \log \Lambda$& AIC \\
     \hline
    Full exponential  &  $\alpha$ & 0.862 & 0.862 & - & 1848.786\\
              & $\beta$ & - & 0.161 & &\\
              & $\gamma$ & - & 0.772 & &\\
              & $\delta$ & - & $\approx0$ & &\\
              & $\rho$ &- & 0.202 & & \\
              \hline
    Case I ($\beta = 1$) & $\alpha$ & 0.862 & 0.862 & 2.082 & 1848.868\\
              & $\gamma$ &  0.064 & 0.073 &  &\\
              & $\delta$ & 0.006 & $\approx0$ & &\\
              & $\rho$ & 0.221 & 0.245 & & \\
              \hline
    Case II ($\gamma = 1$) & $\alpha$ & 0.862 & 0.862 & 0.082 & 1846.868\\
              & $\beta$ & 0.159 & 0.143 & &\\
              & $\delta$ & $\approx0$ & $\approx0$ & &\\
              & $\rho$ & 0.218 & 0.194 & & \\
              \hline
Case III ($\delta = 0$) & $\alpha$ & 0.862 & 0.862 & 0.131 & 1846.918\\
              & $\beta$ & - & 0.134 & &\\
              & $\gamma$ & - & 1.105 & &\\
              & $\rho$ & - & 0.191 & & \\
              \hline
Case IV ($\beta= 1 \ \& \ \gamma = 1$) & $\alpha$ & 0.862 & 0.862 & 300.679 & 2145.465\\
              & $\delta$ &- & $\approx0$ & &\\
              & $\rho$ & - & 0.455 & & \\
              \hline
Case V ($\delta= 0 \ \& \ \gamma = 1$) & $\alpha$ & 0.862 & 0.862 & 0.056 & $\textbf{1844.842}$\\
              & $\beta$ & 0.138 & 0.143 & &\\
              & $\rho$ & 0.191 & 0.194 & & \\
              \hline
    \end{tabular}
    \caption{m.m.e's, m.l.e's, $-2 \log \Lambda$'s (where applicable) and AIC's for mirrored exponential model and sub-models.}
    \label{table: traffic mirror exp}
\end{table}

\begin{table}[H]
    \centering
    \begin{tabular}{cccccc}
    \hline
     Model & Parameter & m.m.e & m.l.e  & $-2 \log \Lambda$& AIC \\
     \hline
    Full Lomax  &  $\alpha'$ & 0.862 & 0.862 & - & 1850.802\\
              & $\beta'$ & - & 0.275 & &\\
              & $\gamma'$ & - & 0.897 & &\\
              & $\delta'$ & - & $\approx0$ & &\\
              & $\eta'$ & - & 0.510 & &\\
              & $\rho$ &- & - & & \\
              \hline
 Lomax($\eta' = 1, \gamma'$)  &  $\alpha'$ & 0.862 & 0.862 & $\approx 0$ & 1848.793\\
              & $\beta'$ & - & 0.217 & &\\
              & $\gamma'$ & - & 1.490 & &\\
              & $\delta'$ & - & $\approx0$ & &\\
              & $\rho$ &- & 0.202 & & \\
              \hline              
Case I ($\beta' = 1$) & $\alpha'$ & 0.862 & 0.862 & 1.545 & 1848.347 \\
              & $\gamma'$ & 14.497 & 12.765 &  &\\
              & $\delta'$ & 0.005 & $\approx0$ & &\\
              & $\rho$ & 0.219 & 0.240 & & \\
              \hline
Case II ($\gamma' = 1$) & $\alpha'$ & 0.862 & 0.862 & 0.066 & 1846.868\\
              & $\beta'$ & 0.203 & 0.181 & &\\
              & $\delta'$ & $\approx 0$ & $\approx0$ & &\\
              & $\rho$ & 0.218 & 0.194 & & \\
              \hline
Case III ($\delta' = 0$) & $\alpha'$ & 0.862 & 0.862 & $\approx 0$ & 1846.793\\
              & $\beta'$ & - & 0.217 & &\\
              & $\gamma'$ & - & 1.491 & &\\
              & $\rho$ & - & 0.202 & & \\
              \hline
Case IV ($\beta' = 1 \ \& \ \gamma' = 1$) & $\alpha'$ & 0.862 & 0.862 & 207.514 & 2052.317\\
              & $\delta'$ & - & $\approx0$ & &\\
              & $\rho$ & - & 0.414 & & \\
              \hline
Case V ($\delta' = 0 \ \& \ \gamma' = 1$) & $\alpha'$ & 0.862 & 0.862 & 0.065 & $\textbf{1844.868}$\\
              & $\beta'$ & 0.176 & 0.181 & &\\
              & $\rho$ & 0.190 & 0.194 & & \\
              \hline
    \end{tabular}
    \caption{m.m.e's, m.l.e's, $-2 \log \Lambda$'s (where applicable) and AIC's for the mirrored Lomax model and sub-models.}
    \label{table: traffic mirror lomax}
\end{table}

\begin{figure}[H]
    \centering
    \includegraphics[width=0.5\linewidth]{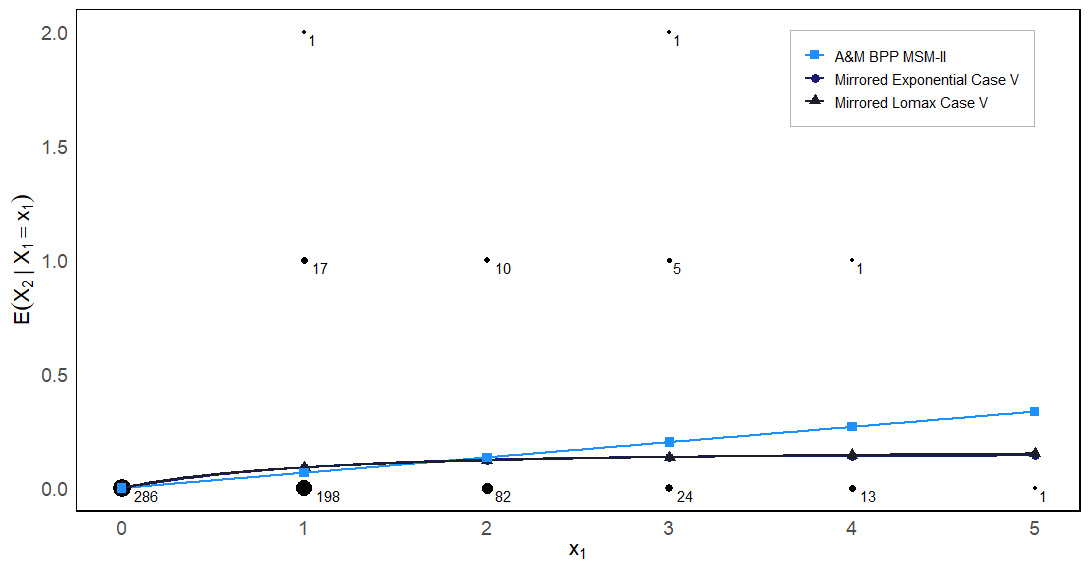}
    \caption{$E(X_2\mid X_1= x_1)$ vs. $X_1 = x_1$ for best fitted models using their m.l.e's. Size of black dots represent relative frequency of observation.}
    \label{fig: traffic mirrored}
\end{figure}

\begin{figure}[H]
    \centering
\animategraphics[controls,autoplay,loop,width= 0.5\textwidth]{1}{Hists/Traffic/}{1}{7}
\caption{Mass functions of best fitted models including Accident and Fatality Data.}
\label{fig: hists traffic}
\end{figure}

\section{Application: Data set II}
We consider the data set that is used in \citet{islam2017analysis}, where the source of the data is from the tenth
wave of Health and Retirement Study (HRS). Here, $X_1$ represents the number of conditions ever had which has a Fisher Dispersion Index of $0.801$, and $X_2$ represents utilization of healthcare
services which has a Fisher Dispersion Index of $1.03$. The Pearson Correlation Coeffcient for
the data is $0.063$.\\

The m.m.e.'s, m.l.e.'s, AIC values, and the corresponding \(-2 \log \Lambda\) statistics (where applicable) for the various models are reported in Table~\ref{table: health exp} and Table~\ref{table: health lomax}. Table~\ref{table: health exp} shows that all exponential sub-models differ significantly from the full exponential model, as evidenced by their respective \(-2 \log \Lambda\) values, with worsened AIC values confirming the inferior fit of these sub-models relative to the full model. A similar conclusion follows from Table~\ref{table: health lomax}, where all Lomax sub-models - except for the Lomax\((\eta' = 1, \gamma')\) and Lomax Case~III sub-models - exhibit \(-2 \log \Lambda\) values indicating significant differences from the full Lomax model.\\

For this dataset, we observe that whenever \(X_1 = 0\), it follows that \(X_2 = 0\), and never \(X_2 > 0\). Consequently, the models naturally identify a no-intercept parameterization as appropriate, yielding estimates \(\tilde{\delta} = 0\) or \(\tilde{\delta}' = 0\).\\

The best-fitting models, according to the AIC criterion, are the full exponential model and the Lomax Case~III sub-model, as shown in Figure~\ref{fig: health non-mirrored}. For comparison, \citet{arnold2021statistical}'s best-fitting model - the Bivariate Pseudo-Poisson Full Model (BPP FM) with an AIC value of \(32772.08\) - is also displayed. Notably, the full exponential model yields \(\hat{\gamma} \approx \infty\), while the Lomax Case~III sub-model yields \(\hat{\gamma}' \approx 0\). This result implies that the two models are nearly identical, consistent with the previously established theoretical equivalence: where \(\lim_{\hat{\gamma} \to \infty}\) in the exponential model is equivalent to  \(\lim_{\hat{\gamma}' \to 0}\) in the Lomax model. Theoretically it was established that such limits would produce independence models where \(X_2\) is no longer conditioned on \(X_1\), since \(X_1 \sim \text{Poisson}(\hat{\alpha})\) and \(X_2 \mid X_1 \sim \text{Poisson}(\hat{\delta} + \hat{\beta})\). However, this independence only arises when \(\hat{\gamma} \to \infty\) or \(\hat{\gamma}' \to 0\) exactly. Empirically, we actually computed \(\hat{\gamma} \approx 2\times 10^6\) and \(\hat{\gamma}' \approx 7\times 10^{-6}\). These values indicate that, at \(x_1 = 0\), the regression function is \(E(X_2 \mid X_1 = x_1) \approx\hat{\delta}\), rather than \(E(X_2 \mid X_1 = x_1) = \hat{\delta} + \hat{\beta}\) under the theoretical limits. Figure~\ref{fig: health non-mirrored} confirms this result, showing \(E(X_2 \mid X_1 = x_1) \approx \hat{\delta}\) (with \(\hat{\delta} = 0\)) for both the full exponential model and the Lomax Case~III sub-model when \(x_1 = 0\). For \(x_1 = 1, \ldots, 8\), the regression function appears constant with \(E(X_2 \mid X_1 = x_1) \approx \hat{\delta} + \hat{\beta}\), illustrating the models' ability to act as a linear model while correctly handling the \((x_1, x_2) = (0, 0)\) observations.\\

Furthermore, all models reported in Table~\ref{table: health exp} and Table~\ref{table: health lomax} outperform the Bivariate COM-Poisson (BCMP) model of \citet{arnold2021statistical}, which was their best-fitting model for the HRS dataset with an AIC value of \(32690.18\). We attribute this improvement to the exponential and Lomax models' capacity to accurately accommodate the \((x_1, x_2) = (0, 0)\) observations.

\begin{table}[H]
    \centering
    \begin{tabular}{cccccc}
    \hline
     Model & Parameter & m.m.e & m.l.e  & $-2 \log \Lambda$& AIC \\
     \hline
    Full exponential  &  $\alpha$ & 2.643 & 2.643 & - & $\textbf{32326.411}$ \\
              & $\beta$ & - & 0.813 & &\\
              & $\gamma$ & - & $\approx \infty$ & &\\
              & $\delta$ & - & $\approx0$ & &\\
              & $\rho$ &- & 0.105 & & \\
              \hline
    Case I ($\beta = 1$) & $\alpha$ & 2.643 & 2.643 & 136.976 & 32461.387 \\
              & $\gamma$ & 0.032 & 1.238 &  &\\
              & $\delta$ & 0.689 & $\approx0$ & &\\
              & $\rho$ & 0.054 & 0.185 & & \\
              \hline
    Case II ($\gamma = 1$) & $\alpha$ & 2.643 & 2.643 & 118.883 & 32443.294\\
              & $\beta$ & 0.258 & 0.924 & &\\
              & $\delta$ & 0.560 & $\approx0$ & &\\
              & $\rho$ & 0.057 & 0.199 & & \\
              \hline
    Case III ($\delta = 0$) & $\alpha$ & 2.643 & 2.643 & 95.285 & 32419.696\\
              & $\beta$ & - & 0.905 & &\\
              & $\gamma$ & - & 1.105 & &\\
              & $\rho$ & - & 0.187 & & \\
              \hline
    Case IV ($\beta = 1 \ \& \ \gamma = 1$) & $\alpha$ & 2.643 & 2.643 & 146.130 & 32468.541\\
              & $\delta$ & $\approx0$ & $\approx0$ & &\\
              & $\rho$ & 0.212 & 0.206 & & \\
              \hline
    Case V ($\delta = 0 \ \& \ \gamma = 1$) & $\alpha$ & 2.643 & 2.643 & 118.883 & 32441.294\\
              & $\beta$ & 0.947 & 0.924 & &\\
              & $\rho$ & 0.201 & 0.199 & & \\
              \hline
    \end{tabular}
    \caption{m.m.e's, m.l.e's, $-2 \log \Lambda$'s (where applicable) and AIC's for exponential model and sub-models.}
    \label{table: health exp}
\end{table}

\begin{table}[H]
    \centering
    \begin{tabular}{cccccc}
    \hline
     Model & Parameter & m.m.e & m.l.e  & $-2 \log \Lambda$& AIC \\
     \hline
    Full Lomax  &  $\alpha'$ & 2.643 & 2.643 & - & 32328.43 \\
              & $\beta'$ & - & 0.813 & &\\
              & $\gamma'$ & - & 0.002 & &\\
              & $\delta'$ & - & $\approx0$ & &\\
              & $\eta'$ & - & 24.324 & &\\
              & $\rho$ &- & - & & \\
              \hline
Lomax($\eta' = 1, \gamma')$  &  $\alpha'$ & 2.643 & 2.643 & $\approx 0$ & 32326.411 \\
              & $\beta'$ & - & 0.813 & &\\
              & $\gamma'$ & - & $\approx0$ & &\\
              & $\delta'$ & - & $\approx0$ & &\\
              & $\rho$ &- & 0.105 & & \\
              \hline
    Case I ($\beta = 1$) & $\alpha'$ & 2.643 & 2.643 & 53.581 & 32378.013 \\
              & $\gamma'$ & $\approx0$ & 0.385 &  &\\
              & $\delta'$ & $\approx 0$ & $\approx0$ & &\\
              & $\rho$ & 0.127 & 0.167 & & \\
              \hline
    Case II ($\gamma' = 1$) & $\alpha'$ & 2.643 & 2.643 & 146.594 & 32471.027 \\
              & $\beta'$ & 0.289 & 1.161 & &\\
              & $\delta'$ & 0.582 & $\approx0$ & &\\
              & $\rho$ & 0.057 & 0.222 & & \\
              \hline
    Case III ($\delta' = 0$) & $\alpha'$ & 2.643 & 2.643 & $\approx 0$ & $\textbf{32324.411}$ \\
              & $\beta'$ & - & 0.813 & &\\
              & $\gamma'$ & - & $\approx0$ & &\\
              & $\rho$ & - & 0.105 & & \\
              \hline
    Case IV ($\beta' = 1 \ \& \ \gamma' = 1$) & $\alpha'$ & 2.643 & 2.643 & 174.438 & 32496.871 \\
              & $\delta'$ & 0.121 & 0.088 & &\\
              & $\rho$ & 0.191 & 0.195 & & \\
              \hline
    Case V ($\delta' = 0 \ \& \ \gamma' = 1$) & $\alpha'$ & 2.643 & 2.643 & 146.594 & 32469.03\\
              & $\beta'$ & 1.186 & 1.161 & &\\
              & $\rho$ & 0.224 & 0.222 & & \\
              \hline
    \end{tabular}
    \caption{m.m.e's, m.l.e's, $-2 \log \Lambda$'s (where applicable) and AIC's for Lomax model and sub-models.}
    \label{table: health lomax}
\end{table}

\begin{figure}[H]
    \centering
    \includegraphics[width=0.5\linewidth]
    {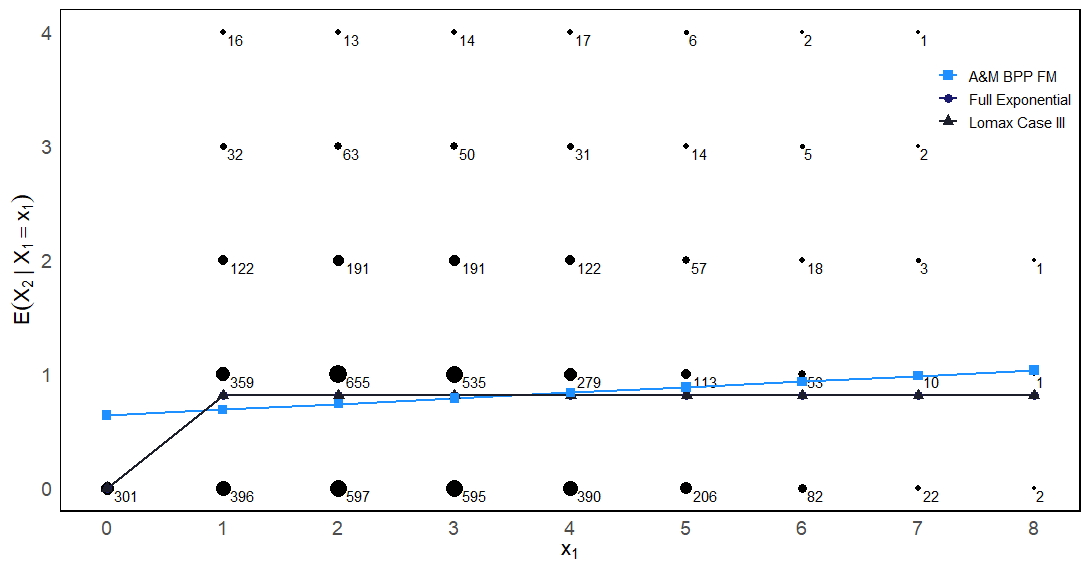}
    \caption{$E(X_2\mid X_1= x_1)$ vs. $X_1 = x_1$ for best fitted models using their m.l.e's. Size of black dots represent relative frequency of observation.}
    \label{fig: health non-mirrored}
\end{figure}

\subsection{Mirrored}
The m.m.e.'s, m.l.e.'s, AIC values, and the corresponding \(-2 \log \Lambda\) statistics (where applicable) for the various mirrored models are reported in Table~\ref{table: health mirror exp} and Table~\ref{table: health mirror lomax}. Figure~\ref{fig: health mirrored} presents the best-fitting models: the mirrored exponential Case~I sub-model and the mirrored Lomax Case~I sub-model, alongside \citet{arnold2021statistical}'s Bivariate Pseudo-Poisson Mirrored Full Model (BPP MFM) with an AIC value of \(32783.080\), which was the best-fitting model for the mirrored dataset in \citet{arnold2021statistical}. Figure~\ref{fig: health mirrored} further shows that this study's best-fitting models are nearly identical to \citet{arnold2021statistical}'s BPP MFM, a linear model, indicating that the exponential and Lomax models considered in this study can approximate linear behavior when appropriate. This, in turn, suggests that a linear specification provides an adequate fit for the mirrored dataset.\\

Additionally, the bivariate probability mass functions for the best-fitting models - computed using Equation~\ref{eq:P(X1, X2) exp} for the exponential model and Equation~\ref{eq:P(X1, X2) lomax} for the Lomax model - are presented for the HRS dataset in Figure~\ref{fig: hists traffic}.

\begin{table}[H]
    \centering
    \begin{tabular}{cccccc}
    \hline
     Model & Parameter & m.m.e & m.l.e  & $-2 \log \Lambda$& AIC \\
     \hline
    Full exponential  &  $\alpha$ & 0.769 & 0.769& - & 32784.617 \\
              & $\beta$ & - & 0.501 & &\\
              & $\gamma$ & - & 0.287  & &\\
              & $\delta$ & - & 2.555 & &\\
              & $\rho$ &- & 0.056 & & \\
              \hline
    Case I ($\beta = 1$) & $\alpha$ & 0.769 & 0.769 & 0.144 & $\textbf{32782.761}$ \\
              & $\gamma$ & 0.126 & 0.121 &  &\\
              & $\delta$ & 2.556 & 2.559 & &\\
              & $\rho$ & 0.058 & 0.056 & & \\
              \hline
    Case II ($\gamma = 1$) & $\alpha$ &0.769 & 0.769 & 1.227 & 32783.844 \\
              & $\beta$ & 0.271 & 0.240 & &\\
              & $\delta$ & 2.538 & 2.551 & &\\
              & $\rho$ & 0.057 & 0.050 & & \\
              \hline
    Case IV ($\beta = 1 \ \& \ \gamma = 1$) & $\alpha$ & 0.769 & 0.769 & 165.397 & 32946.015\\
              & $\delta$ & 2.258 & 2.298 & &\\
              & $\rho$ & 0.205 & 0.203 & & \\
              \hline
    \end{tabular}
    \caption{m.m.e's, m.l.e's, $-2 \log \Lambda$'s (where applicable) and AIC's for mirrored exponential model and sub-models.}
    \label{table: health mirror exp}
\end{table}

\begin{table}[H]
    \centering
    \begin{tabular}{cccccc}
    \hline
     Model & Parameter & m.m.e & m.l.e  & $-2 \log \Lambda$& AIC \\
     \hline
    Full Lomax  &  $\alpha'$ & 0.769 & 0.769& - & 32786.649 \\
              & $\beta'$ & - & 1.517 & &\\
              & $\gamma'$ & - & 3.746  & &\\
              & $\delta'$ & - & 2.556 & &\\
            & $\eta'$ & - & 0.361 & &\\
              & $\rho$ &- & - & & \\
              \hline
    Lomax($\eta' = 1, \gamma'$)  &  $\alpha'$ & 0.769 & 0.769& $\approx0$ & 32784.639 \\
              & $\beta'$ & - & 0.869 & &\\
              & $\gamma'$ & - & 5.980  & &\\
              & $\delta'$ & - & 2.556 & &\\
              & $\rho$ &- & 0.056 & & \\
              \hline
    
    Case I ($\beta' = 1$) & $\alpha'$ & 0.769 & 0.769 & $\approx0$ & $\textbf{32782.647}$ \\
              & $\gamma'$ & 7.052 & 7.181 &  &\\
              & $\delta'$ & 2.555 & 2.556 & &\\
              & $\rho$ & 0.057 & 0.056 & & \\
              \hline
    Case II ($\gamma' = 1$) & $\alpha'$ &0.769 & 0.769 & 1.319 & 32783.967 \\
              & $\beta'$ & 0.346 & 0.305 & &\\
              & $\delta'$ & 2.538 & 2.551 & &\\
              & $\rho$ & 0.057 & 0.050 & & \\
              \hline
    Case IV ($\beta' = 1 \ \& \ \gamma' = 1$) & $\alpha'$ & 0.769 & 0.769 & 86.175 & 32866.824\\
              & $\delta'$ & 2.340 & 2.364 & &\\
              & $\rho$ & 0.162 & 0.161 & & \\
              \hline
    \end{tabular}
    \caption{m.m.e's, m.l.e's, $-2 \log \Lambda$'s (where applicable) and AIC's for mirrored Lomax model and sub-models.}
    \label{table: health mirror lomax}
\end{table}

\begin{figure}[H]
    \centering
    \includegraphics[width=0.5\linewidth]{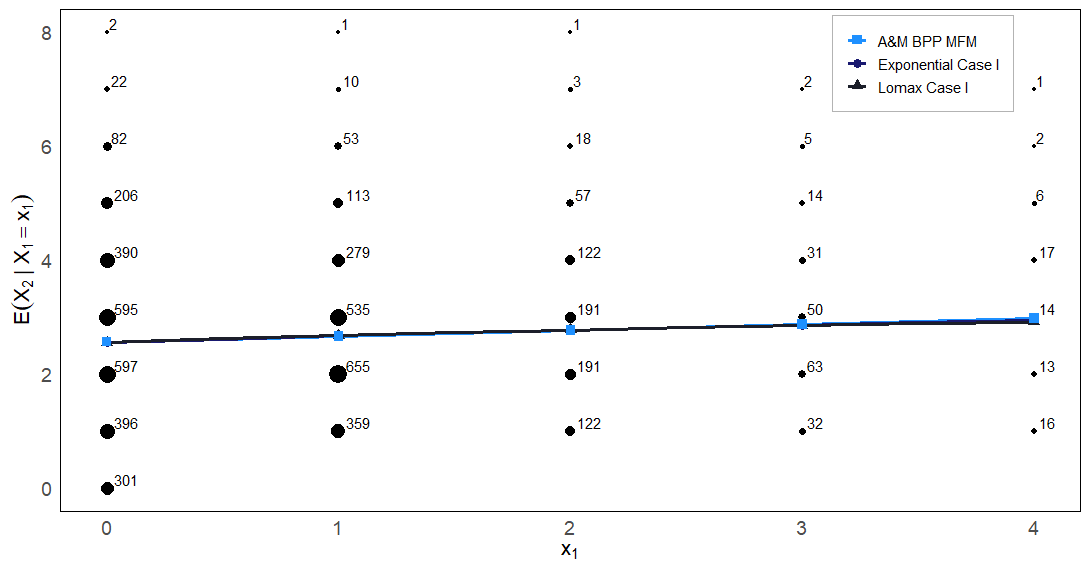}
    \caption{$E(X_2\mid X_1= x_1)$ vs. $X_1 = x_1$ for best fitted models using their m.l.e's. Size of black dots represent relative frequency of observation.}
    \label{fig: health mirrored}
\end{figure}

\begin{figure}[H]
    \centering
\animategraphics[controls,autoplay,loop,width= 0.5\textwidth]{1}{Hists/Health/}{1}{7}
\caption{Mass functions of best fitted models including HRS Data.}
\label{fig: hists health}
\end{figure}

\section{Conclusion}
This study set out to develop models capable of representing negative correlation between \(X_1\) and \(X_2\) by introducing curvature into the conditional rate \(\lambda_2(x_1)\). On the HRS dataset, the proposed models modeled the boundary observation \((x_1,x_2)=(0,0)\) more adequately than \citet{arnold2021statistical}'s linear models, while retaining the capacity to approximate linear behavior where appropriate. Additionally, the accident and fatality data indicated why added curvature is deemed necessary and may lead to improved model fit.\\

Caveats arose from the additional complexity of this curvature - including imposed theoretical bounds for the study's models' correlations, and data specifications for m.m.e existence. The application results further revealed that parameter estimation frequently led to near-equivalence across models - as evidenced in their respective regression plots.\\

Future work should assess the performance of the study's models on negatively correlated bivariate count data and incorporate formal goodness-of-fit testing as extensively covered by \citet{veeranna2023goodness}.

\appendix

\section*{Appendix}\label{app: thms}
\begin{theorem} \label{thm: f(v) strict inc}
$f(\tilde{\nu}) = \frac{\tilde{\alpha}^2(\tilde{\nu} -1)^2}{\tilde{\mu}^{(\tilde{\nu} - 1)^2} -1 }$ is strictly increasing for $\tilde{\nu} \in (0, 1)$.
\begin{proof}
    Note the derivative of $f(\tilde{\nu})$ with respect to $\tilde{\nu}$:
    \begin{align}
        \frac{\partial}{\partial\nu}\left(\frac{\tilde{\alpha}^2(\tilde{\nu} -1)^2}{{e^{\tilde{\alpha}(\tilde{\nu}- 1)^2}} -1 } \right) & =\frac{2\tilde{\alpha}^2(\tilde{\nu} -1) \left( e^{\tilde{\alpha}(\tilde{\nu}-1)^2}\left(1 -\tilde{\alpha}(\tilde{\nu} -1)^2\right)-1 \right)}{\left( e^{\tilde{\alpha}(\tilde{\nu}-1)^2}- 1\right)^2}. \nonumber
    \end{align}
    Since $0<\tilde{\nu}<1 \implies -2\alpha^2<2\alpha^2(\tilde{\nu} -1) < 0$. It remains to be shown that $f(t) = e^t(1 -t) - 1 <0 $ for $t = \tilde{\alpha}(\tilde{\nu} - 1)^2$ noting $0<t<\alpha$, to ultimately show $\frac{\partial f(\tilde{\nu}) }{\partial \tilde{\nu}}> 0$. Now since $\frac{d}{dt} \left(e^t(1-t) -1\right) = -te^t$, $f(t)$ is monotonic for $t>0$ and since $\lim_{t \to 0^+}f(t) = 0$, we have shown $f(t) < 0$ for $t>0$. It follows that since $\frac{\partial f(\tilde{\nu}) }{\partial \tilde{\nu}} >0$ for $0<\tilde{\nu}<1$, $f(\tilde{\nu})$ is strictly increasing on this interval.
\end{proof}
\end{theorem}

\begin{theorem} \label{thm: f(gamma') strict inc}
  $f(\tilde{\gamma}') = \frac{(\tilde{\alpha}')^2\left[{}_{1}F_{1}(\tilde{\gamma}') - \left(\frac{\tilde{\gamma}'}{\tilde{\gamma}'+1}\right){}_{1}F_{1}(\tilde{\gamma}'+1)  \right]^2}{\tilde{\mu}' {}_{2}F_{2}(\tilde{\gamma}') -  [{}_{1}F_{1}(\tilde{\gamma}')]^2 } $ is strictly increasing for $\tilde{\gamma}' > 0$.
  \begin{proof}
     We aim to show $ \frac{\partial f(\tilde{\gamma}')}{\partial\tilde{\gamma}'}> 0$ for $\tilde{\gamma}' > 0$. Now, through substitution of Equations \ref{eq: moment equiv 1 hyper}, \ref{eq: moment equiv 2 hyper} and \ref{eq: moment equiv 2 hyper} into $f(\tilde{\gamma}')$, we have $f(\tilde{\gamma'}) = f(\tilde{\nu})$ (where $f(\tilde{\nu})$ is referenced from Theorem \ref{thm: f(v) strict inc}). Subsequently, Theorem \ref{thm: f(v) strict inc} proved $\frac{\partial f(\tilde{\nu})}{\partial{\tilde{\nu}}} > 0$ for $0<\tilde{\nu}<1$ (or equivalently, for $\tilde{\gamma}>0$). Now since:
     \begin{align}
         \frac{\partial f(\tilde{\gamma}')}{\partial\tilde{\gamma}'} & = \frac{\partial f(\tilde{\gamma}')}{\partial\tilde{\nu}} \cdot \frac{\partial \tilde{\nu} }{\partial \tilde{\gamma}'}\nonumber \\
         & = \frac{\partial f(\tilde{\nu})}{\partial\tilde{\nu}} \cdot \frac{\partial \tilde{\nu} }{\partial \tilde{\gamma}'},\nonumber 
     \end{align}
     it remains to be shown that $\frac{\partial\tilde{\nu}}{\partial\tilde{\gamma}'} > 0$. Now from Equation \ref{eq: moment equiv 1 hyper} and noting we already established $\tilde{\alpha} = \tilde{\alpha}'$, we have:
     \begin{align}
         \tilde{\nu} &= \frac{1}{\tilde{\alpha}}\log \left({}_{1}F_{1}(\tilde{\gamma}', \tilde{\gamma}'+1;\tilde{\alpha'}) \right) \nonumber \\
         & = \frac{1}{\tilde{\alpha}} \log\left(\sum_{i = 0}^\infty\frac{\tilde{\gamma}'}{\tilde{\gamma}'+i} \frac{\tilde{\alpha}}{i!} \right). \nonumber
     \end{align}
     Now, taking the derivative of $\tilde{\nu}$ with respect to $\tilde{\gamma}'$, we have:
     \begin{align}
         \frac{\partial\tilde{\nu}}{\partial \tilde{\gamma}'} & = \frac{1}{\tilde{\alpha}} \frac{1}{\sum_{i = 0}^\infty\frac{\tilde{\gamma}'}{\tilde{\gamma}'+i} \frac{\tilde{\alpha}}{i!}}\sum_{i = 0}^\infty\frac{i}{\left(\tilde{\gamma}'+i\right)^2} \frac{\tilde{\alpha}}{i!} \nonumber \\
         & > 0. \nonumber
     \end{align}
  \end{proof}
\end{theorem}

\begin{lemma} \label{lemma: limit f(gamma)}
   $\lim_{\zeta \to 0}{}_{\psi}F_{\psi}\left(\zeta, \ldots, \zeta, (\zeta +1), \ldots, (\zeta+1);\xi \right) = 1$ and $\lim_{\zeta \to 0}{}_{\psi}F_{\psi}\left((\zeta+1), \ldots, (\zeta+1), (\zeta +2), \ldots, (\zeta+2);\xi \right) = \frac{1}{\xi}\left(e^\xi - 1 \right){}$ for $\zeta, \xi >0$ and $\psi \in \mathbb{N}$.
\end{lemma}
\begin{proof}
    We note:
    \begin{align}
        \lim_{\zeta \to 0} \left\{ _{\psi}F_{\psi}\left(\zeta, \ldots, \zeta, (\zeta +1), \ldots, (\zeta+1);\xi \right) \right\} &= \lim_{\zeta \to 0} \left\{ \sum_{i = 0}^\infty \left[\frac{\zeta}{\zeta +i}\right]^\psi \frac{\xi^i}{i!} \right\}  \nonumber\\
        &= \lim_{\zeta \to 0} \left\{1+ \left(\frac{\zeta}{\zeta+1} \right)^\psi \xi + \left(\frac{\zeta}{\zeta+2} \right)^\psi \frac{\xi^2}{2}+ \left(\frac{\zeta}{\zeta+3} \right)^\psi \frac{\xi^3}{6} +\ldots \right \}\nonumber \\
        & = 1, \nonumber
    \end{align}
and:
\begin{align}
      \lim_{\zeta \to 0} \left\{  {}_{\psi}F_{\psi}\left( (\zeta +1), \ldots, (\zeta +1), (\zeta +2), \ldots, (\zeta+2);\xi \right) \right\}
      &=  \lim_{\zeta \to 0} \left\{ \sum_{i = 0}^\infty \left[\frac{\zeta+1}{\zeta +1+i}\right]^\psi \frac{\xi^i}{i!} \right\}\nonumber\\
    &=  \lim_{\zeta \to 0} \left\{ 1+ \left(\frac{\zeta+1}{\zeta+2} \right)^\psi \xi + \left(\frac{\zeta+1}{\zeta+3} \right)^\psi \frac{\xi^2}{2}+ \left(\frac{\zeta+1}{\zeta+4} \right)^\psi \frac{\xi^3}{6} +\ldots \right\}\nonumber\\
& = \lim_{\zeta \to 0} \left\{ 1+ \left(\frac{1}{2} \right)^\psi \xi + \left(\frac{1}{3} \right)^\psi \frac{\xi^2}{2}+ \left(\frac{1}{4} \right)^\psi \frac{\xi^3}{6} +\ldots \right\}\nonumber\\
& = \lim_{\zeta \to 0} \left\{ \sum_{i = 0}^\infty \frac{1}{1+i} \frac{\xi^i}{i!} \right\}\nonumber \\
&= \lim_{\zeta \to 0} \left\{\frac{1}{\xi} \sum_{i = 0}^\infty \frac{\xi^{i+1}}{(i+1)!} \right\}\nonumber \\
& = \lim_{\zeta \to 0} \left\{\frac{1}{\xi} \left(\sum_{i = 0}^\infty \frac{\xi^{i}}{i!} - 1\right) \right\} \nonumber \\ 
& =\frac{1}{\xi} \left(e^\xi  - 1\right). \nonumber
\end{align}
  
\end{proof}
\begin{theorem}\label{thm: limit gamma lomax}
    $\lim_{\tilde{\gamma}' \to 0^+} \left\{ \frac{(\tilde{\alpha}')^2\left[{}_{1}F_{1}(\tilde{\gamma}') - \left(\frac{\tilde{\gamma}'}{\tilde{\gamma}'+1}\right){}_{1}F_{1}(\tilde{\gamma}'+1)  \right]^2}{\tilde{\mu}' {}_{2}F_{2}(\tilde{\gamma}') -  [{}_{1}F_{1}(\tilde{\gamma}')]^2 } \right\} = \frac{\left(\tilde{\alpha}'\right)^2}{\tilde{\mu}' - 1}$ for $\tilde{\alpha}' > 0$.
\end{theorem}
\begin{proof}
    We note:
    \begin{align}
   &\lim_{\tilde{\gamma}' \to 0^+} \left\{ \frac{(\tilde{\alpha}')^2\left[{}_{1}F_{1}(\tilde{\gamma}') - \left(\frac{\tilde{\gamma}'}{\tilde{\gamma}'+1}\right){}_{1}F_{1}(\tilde{\gamma}'+1)  \right]^2}{\tilde{\mu}' {}_{2}F_{2}(\tilde{\gamma}') -  [{}_{1}F_{1}(\tilde{\gamma}')]^2 } \right\} \nonumber\\
     &=\frac{(\tilde{\alpha}')^2\left[\lim_{\tilde{\gamma}' \to 0^+} \left\{{}_{1}F_{1}(\tilde{\gamma}')\right\} -\left(\lim_{\tilde{\gamma}' \to 0^+}\left\{\frac{\tilde{\gamma}'}{\tilde{\gamma}'+1}\right\}\right) \left(\lim_{\tilde{\gamma}'\to 0^+} \left\{ {}_{1}F_{1}(\tilde{\gamma}'+1) \right\} \right) \right]^2}{\tilde{\mu}' \lim_{\tilde{\gamma}' \to 0^+} \left\{{}_{2}F_{2}(\tilde{\gamma}')\right\} -  [\lim_{\tilde{\gamma}'\to 0^+} \left\{{}_{1}F_{1}(\tilde{\gamma}')\right\}]^2 }, \nonumber
    \end{align}
and by using Lemma \ref{lemma: limit f(gamma)} through direct substitution, we obtain $\frac{\left(\tilde{\alpha}'\right)^2}{\tilde{\mu}' - 1}$. 
\end{proof}

\bibliographystyle{unsrtnat}
\bibliography{ref}

\end{document}